\documentclass[reqno]{amsart}
\usepackage{url}
\usepackage{amssymb,amsthm,amsfonts,amstext,amsmath}
\usepackage{newcent}         
\usepackage{helvet}          
\usepackage{courier}         
\usepackage{graphicx}
\usepackage{enumerate}
\numberwithin{equation}{section}
\usepackage{color}
\usepackage{float}
\usepackage[colorlinks=true,linkcolor=blue,citecolor=magenta]{hyperref}
\usepackage{mathrsfs}
\usepackage{tikz}
\usepackage{dsfont}

\usepackage[utf8]{inputenc}
\usepackage[T1]{fontenc}
\usetikzlibrary{backgrounds}
\usetikzlibrary{patterns,fadings}
\usetikzlibrary{arrows,decorations.pathmorphing}
\usetikzlibrary{calc}
\definecolor{light-gray}{gray}{0.95}
\usepackage{float}
\newtheorem{theorem}{Theorem}[section]
\newtheorem{lemma}[theorem]{Lemma}

\newtheorem{proposition}[theorem]{Proposition}
\newtheorem{corollary}[theorem]{Corollary}
\newtheorem{remark}[theorem]{Remark}
\newtheorem{definition}{Definition}
\newcommand{\mc}[1]{{\mathcal #1}}

\newcommand{\bb}[1]{{\mathbb #1}}

\newcommand{\eps}{\varepsilon}

\newcommand{\p}{\partial}
\newcommand{\N}{\mathbb N}
\newcommand{\R}{\mathbb R}

\newcommand{\Z}{\mathbb Z}
\newcommand{\ZZ}{{\bf Z}}
\newcommand{\E}{\mathbb E}

\newcommand{\PP}{\mathbb P}
\newcommand{\PSNOB}{P_t^{\text{\tiny\rm snob}}}
\newcommand{\PROBIN}{P_t^{\text{\tiny\rm Robin}}}
\newcommand{\GROBIN}{G^{\text{Robin}}}
\newcommand{\BSNOB}{B^{\text{\tiny\rm snob}}}
\newcommand{\BREF}{B^{\text{\tiny\rm ref}}}
\newcommand{\PREF}{P_t^{\text{\tiny\rm ref}}}
\newcommand{\XREF}{X_t^{\text{\tiny\rm ref}}}

\newcommand{\XSLOW}{X^{\text{\tiny\rm slow}}}
\newcommand{\ZREF}{\xi^{\text{\tiny\rm ref}}}
\newcommand{\XSB}{X^{\text{\tiny\rm slow}}}
\newcommand{\BSB}{B^{\text{\tiny\rm slow}}}
\newcommand{\YSB}{Y^{\text{\tiny\rm slow}}}
\newcommand{\BSBhat}{\hat{B}^{\text{\tiny\rm slow}}}

\newcommand{\CSNOB}{\mathscr{C}_{\text{\rm b}}(\bb G)}
\newcommand{\CSNOBO}{\mathscr{C}_{0}(\mathbb G)}
\newcommand{\LSNOB}{\text{\rm BL}(\beta)}
\newcommand\topo[2]{\genfrac{}{}{0pt}{}{#1}{#2}}
\DeclareMathOperator{\sgn}{sgn}
\def\centerarc[#1](#2)(#3:#4:#5){\draw[#1] ($(#2)+({#5*cos(#3)},{#5*sin(#3)})$) arc (#3:#4:#5);}
\newcommand{\pfrac}[2]{\genfrac{}{}{}{1}{#1}{#2}}

\newcommand{\dd}{\displaystyle}

\newcommand{\one}{\mathds{1}}

\newcommand{\fodd}{f_{\text{\rm odd}(n)}}
\newcommand{\feven}{f_{\text{\rm even}(n)}}

\newcommand{\fo}{f_{\text{\rm odd}}}
\newcommand{\fe}{f_{\text{\rm even}}}

\newcommand{\bigP}[2]{\big(\pfrac{#1}{#2})}

\def\I{\mathrm{I}}
\def\A{{\bf{A}}}
\def\II{\mathrm{I\kern-0.1emI}}

\let\oldtocsection=\tocsection
\let\oldtocsubsection=\tocsubsection
\let\oldtocsubsubsection=\tocsubsubsection
\renewcommand{\tocsection}[2]{\hspace{0em}\oldtocsection{#1}{#2}}
\renewcommand{\tocsubsection}[2]{\hspace{1em}\oldtocsubsection{#1}{#2}}
\renewcommand{\tocsubsubsection}[2]{\hspace{2em}\oldtocsubsubsection{#1}{#2}}
\DeclareRobustCommand{\SkipTocEntry}[5]{}


\begin{document}

\title[The Slow Bond RW and the  Snapping Out BM]{The Slow Bond Random Walk\\ and the  Snapping Out Brownian Motion}

\author[D. Erhard]{Dirk Erhard}
\address{UFBA\\
 Instituto de Matem\'atica, Campus de Ondina, Av. Adhemar de Barros, S/N. CEP 40170-110\\
Salvador, Brazil}
\curraddr{}
\email{erharddirk@gmail.com}
\thanks{}

\author[T. Franco]{Tertuliano Franco}
\address{UFBA\\
 Instituto de Matem\'atica, Campus de Ondina, Av. Adhemar de Barros, S/N. CEP 40170-110\\
Salvador, Brazil}
\curraddr{}
\email{tertu@ufba.br}
\thanks{}

\author[D. S. da Silva]{Diogo S. da Silva}
\address{UFBA\\
 Instituto de Matem\'atica, Campus de Ondina, Av. Adhemar de Barros, S/N. CEP 40170-110\\
Salvador, Brazil\newline
\indent IFBA, Rua Vereador Romeu Agr\'ario Martins, Tento, S/N. \newline CEP: 45400-000,
Valen\c ca, Brazil}
\curraddr{}
\email{dsdorea@gmail.com}
\thanks{}


\begin{abstract} 
We consider the continuous time  symmetric random walk  with a slow bond on $\bb Z$, which rates are equal to $1/2$ for all bonds, except for the bond  of vertices $\{-1,0\}$, which associated rate is given by $\alpha n^{-\beta}/2$, where $\alpha\geq 0$ and $\beta\in [0,\infty]$ are the parameters of the model. We prove here a functional central limit theorem for the random walk with a slow bond: if $\beta<1$, then it converges to the usual Brownian motion. If $\beta\in (1,\infty]$, then it converges to the reflected Brownian motion. And at the critical value $\beta=1$, it converges to the \textit{snapping out Brownian motion} (SNOB) of parameter $\kappa=2\alpha$, which is a Brownian  type-process recently  constructed in \cite{Lejay}. We also provide Berry-Esseen estimates in the dual bounded Lipschitz metric for the weak convergence of one-dimensional distributions, which we believe to be sharp. 
\end{abstract}

\maketitle


\section{Introduction}\label{s1}

Arguably one of the most important results in probability theory and statistical mechanics is Donsker's theorem which establishes a link between two key objects in the field: random walk and Brownian motion. 

In the literature many Donsker-type theorems can be found; however, most of the results are concerned with limits of random walks (in random environment, in non-Markovian setting, in deterministic non-homogeneous medium etc.) towards the \textit{usual Brownian motion}. A significant smaller set of results are about convergence towards Brownian motion with boundary conditions, see \cite{Amir} for an example.

In this paper, we prove a functional central limit theorem for the \textit{slow bond random walk} (abbreviated \textit{slow bond RW}), which is the continuous time nearest neighbour random walk on $\bb Z$ with jump rates given by $\alpha/(2n^\beta)$ if the jump is along the edge $\{-1,0\}$ and  $1/2$ otherwise.

The jump rates of the slow bond RW are depicted in Figure~\ref{fig1a}. We remark that this process was inspired by the \textit{exclusion process with a slow bond}, see \cite{efgnt,fgn1,fgn2,fgn3,FGSCMP,FN} among others. The slow bond RW can be seen simply as the \textit{symmetric exclusion process with a slow bond}  with a single particle. For the symmetric exclusion process with a slow bond, under certain initial conditions,   \cite{fgn1,fgn2,fgn3} established a dynamical phase transition in $\beta$. Surprisingly the proof of that transition neither implies or uses a similar transition for the slow bond RW nor does it give any indication of how to establish such a result. Yet, it would be natural to expect a dynamical phase transition for the slow bond RW as well. This is exactly the content of this work.
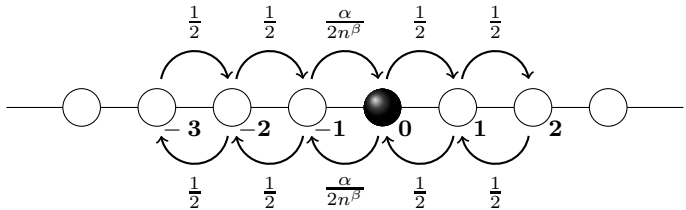
\begin{figure}[!htb]
\centering
\begin{tikzpicture}
\centerarc[thick,<-](1.5,0.3)(10:170:0.45);
\centerarc[thick,->](1.5,-0.3)(-10:-170:0.45);
\centerarc[thick,<-](2.5,0.3)(10:170:0.45);
\centerarc[thick,->](2.5,-0.3)(-10:-170:0.45);
\centerarc[thick,->](3.5,-0.3)(-10:-170:0.45);
\centerarc[thick,<-](3.5,0.3)(10:170:0.45);
\centerarc[thick,->](4.5,-0.3)(-10:-170:0.45);
\centerarc[thick,<-](4.5,0.3)(10:170:0.45);
\centerarc[thick,->](5.5,-0.3)(-10:-170:0.45);
\centerarc[thick,<-](5.5,0.3)(10:170:0.45);

\draw (-1,0) -- (8,0);

\shade[ball color=black](4,0) circle (0.25);

\filldraw[fill=white, draw=black]
(6,0) circle (.25)
(2,0) circle (.25)
(1,0) circle (.25)
(0,0) circle (.25)
(3,0) circle (.25)
(5,0) circle (.25)
(7,0) circle (.25)
;

\draw (1.3,-0.05) node[anchor=north] {\small $\bf - 3 $};
\draw (2.3,-0.05) node[anchor=north] {\small $\bf - 2 $};
\draw (3.3,-0.05) node[anchor=north] {\small $\bf -1$};
\draw (4.3,-0.05) node[anchor=north] {\small $\bf 0$};
\draw (5.3,-0.05) node[anchor=north] {\small $\bf 1$};
\draw (6.3,-0.05) node[anchor=north] {\small $\bf 2$};
\draw (1.5,0.8) node[anchor=south]{$\frac{1}{2}$};
\draw (1.5,-0.8) node[anchor=north]{$\frac{1}{2}$};
\draw (2.5,0.8) node[anchor=south]{$\frac{1}{2}$};
\draw (2.5,-0.8) node[anchor=north]{$\frac{1}{2}$};
\draw (3.5,0.8) node[anchor=south]{$\frac{\alpha}{2n^{\beta}}$};
\draw (4.5,-0.8) node[anchor=north]{$\frac{1}{2}$};
\draw (4.5,0.8) node[anchor=south]{$\frac{1}{2}$};
\draw (5.5,-0.8) node[anchor=north]{$\frac{1}{2}$};
\draw (5.5,0.8) node[anchor=south]{$\frac{1}{2}$};
\draw (3.5,-0.8) node[anchor=north]{$\frac{\alpha}{2n^{\beta}}$};
\end{tikzpicture}
\caption{Jump rates for the \textit{slow bond random walk}}
\label{fig1a}
\end{figure}

We show here that the limit for the slow bond RW depends on the range of $\beta$. If $\beta\in[0,1)$, the limit is the usual Brownian motion (BM), meaning that the slow bond has no effect in the limit; if $\beta\in(1,\infty]$ it is the reflected Brownian motion, meaning that the slow bond is powerful enough to completely split the real line around the origin in the limit. Finally, and most important, in the critical case $\beta=1$, the limit is given by the \textit{snapping out Brownian motion}, which is a stochastic process recently constructed in \cite{Lejay}. This process can be understood as a Brownian motion with the following boundary behaviour: until the moment that the local time at zero reaches a value given by an (independent of the BM) exponential random variable, the process behaves as the reflected BM. At that moment, the process is then restarted, according to an honest coin, in the positive or in the negative half line (at the origin). A precise definition is given in Section~\ref{s2} as well as  a brief explanation  of why the snapping out BM is related with the \textit{partially reflected BM}, see \cite{Grebenkov} on the latter process. 

The partially reflected BM is known to be relevant in many physical situations, including nuclear magnetic resonance, heterogeneous catalysis and electric transport in electrochemistry, see \cite{Grebenkov,NMR_survey} and the same importance is expected for the snapping out BM. Some  methods of simulations  for both the snapping out BM and the partially reflected BM have been described, see \cite[Section 6]{Lejay} and \cite[Subsection 1.1.4]{Grebenkov} and references therein. However, no rigorous functional central limit theorem has been proved until now. Furthermore, the choice of an approximating  model itself was open. Here we present a very simple discrete model which rigorously can be shown to converge to the snapping out BM.
 
A relevant feature of this work is the approach itself: since the slow bond RW cannot be written as a sum of  independent random variables, classical approaches as convergence of characteristic functions,  successive replacements (as in \cite[p.~42]{Bill} for instance) or via the $d_k$ distance (see \cite[Chapter 2]{araujoguine} for instance) do not apply here.   To overcome this difficulty, we deal directly with the convergence of expectation of bounded continuous functions to show the convergence of the one-dimensional distributions. The problem is then translated into a convergence of solutions of a  semi-discrete scheme by looking at  Kolmogorov's equation for the generator.

Convergence of semi-discrete schemes with boundary conditions are often technically very challenging. However, we  avoid here standard techniques of convergence for theses problems. Instead, via the Feynman-Kac formula, we are able to establish convergence of the semi-discrete scheme by means of probabilistic tools. The key observation is that it is possible to rewrite the problem in terms of a simple random walk and a tilted reflected random walk. The main tools developed and used involve  local times, projection of Markov chains, local central limit theorems and symmetry arguments.

The convergence of the finite dimensional distributions turn out to follow more or less directly from the convergence of the one-dimensional distributions.
Tightness issues have been handled through an appropriate application of the Burkholder-Davis-Gundy inequality to the Dynkin  martingale. 
 
 \textit{En passant}, we obtain in Section \ref{Sec:semigroup} an explicit formula for the semigroup of the snapping out BM and characterize it as a solution of a PDE with Robin boundary conditions, which is a small ingredient in the proof, but of interest by itself. A substantial part of this work is dedicated to show Berry-Esseen estimates for the one-dimensional distributions in the dual bounded Lipschitz metric. The convergence rates are indeed slower than in the classical case. A discussion of why this phenomena occurs is presented in Section~\ref{s2}.

We believe that the approach of this paper could be successful in other situations, in particular  to prove functional central limits of random walks in non-homogeneous medium. The philosophy behind our work is that \textit{analytical problems inherited from probabilistic problems are easier solved by probabilistic methods}.

The outline of the paper is the following: In Section~\ref{s2} we present definitions and state results. Section~\ref{Sec:semigroup} is reserved to present the semigroup formula for the snapping out BM. Section~\ref{s3} deals with necessary ingredients in the proof of convergence of one-dimensional distributions and Berry-Esseen estimates, all of them related to local times. Section~\ref{s5} gives the proof of Berry-Esseen estimates in the dual Lipschitz bounded norm and convergence of one-dimensional distributions. Section~\ref{s6} extends the proof to finite-dimensional distributions and in Section~\ref{sec:tight} we prove the tightness of the processes in the $J_1$-Skorohod topology of $\mathscr{D}([0,1], \bb R)$. In Appendix~\ref{appendixA} we review some known results for the sake of completeness.


\section{Statement of results}\label{s2}

\textbf{Notation:} to avoid an overload of notation,  expectations of any process considered in this article starting from a point $x$ will be denoted by $\bb E_x$. Throughout the paper, the symbol $\lesssim$ will mean that the quantity standing on the left hand side of it is smaller than some multiplicative constant times the quantity on the right hand side of it. The proportionality constant may change from one line to another, but it will never depend on the scaling parameter $n\in \bb N$.\medskip

The \textit{slow bond random walk} we define here is the Feller process on $\Z$ denoted by $\{\XSB_t:t\geq 0\}$ whose generator $\mathsf{L}_n$ acts on local functions $f:\bb Z\to \bb R$ via 
\begin{equation}\label{generator}
\mathsf{L}_n f(x) \;=\;  \xi_{x,x+1}^n \Big[f(x+1)-f(x)\Big]+\xi_{x,x-1}^n \Big[f(x-1)-f(x)\Big]\,,
\end{equation}
where 
\begin{equation*}
\xi_{x,x+1}^n\;=\;\xi_{x+1,x}^n\;=\; 
\begin{cases}
\dd\frac{\alpha}{2n^{\beta}},& \text{ if } x=-1,\vspace{4pt} \\
1/2, & \text{ otherwise}. 
\end{cases}
\end{equation*}

The \textit{elastic (or plastic or partially reflected) Brownian motion} on $[0,\infty)$ is a continuous stochastic process which can be understood as an intermediate process between the absorbed Brownian motion and the reflected Brownian motion on $[0,\infty)$. This elastic Brownian motion  can be described as the reflected Brownian motion  killed at a stopping time with exponential distribution: first, for a given positive parameter $\kappa$ we toss a  random variable $Y\sim\exp(\kappa)$ independent of the reflected Brownian motion; once the local time of the reflected Brownian motion at zero reaches $Y$, it is killed (at the origin). We refer the reader to  the survey \cite{Grebenkov} for the connection of the elastic Brownian motion (in the $d$-dimensional setting) and its connections with mixed boundary value problems and Laplacian transport phenomena.
 
The \textit{snapping out Brownian motion} process on $\bb G = (-\infty,0^-]\cup [0^+,\infty)$ with prameter $\kappa$,  abbreviated SNOB, is a Feller process recently constructed in \cite{Lejay} by gluing pieces of the elastic BM of parameter $2\kappa$. Once the $2\kappa$-elastic BM is killed, we decide whether to restart the process in $0^+$ or $0^-$  with probability $1/2$. An equivalent way of defining it is to consider the $\kappa$-elastic BM, but when the process is killed at $0^+$ (equiv. $0^-$), it is restarted on the opposite side $0^-$ (equiv. $0^+$).

Let  $\CSNOB$   be the set of bounded continuous functions $f:\bb G\to \bb R$, which are naturally identified with the set of bounded continuous functions  $f:\bb R\backslash \{0\}\to \bb R$ with side limits at zero. 
Denote by $\CSNOBO\subset \CSNOB$ the set of  bounded, continuous  functions $f:\bb G\to \bb R$  vanishing at infinity. Many statements in this paper can easily be extended to far more general spaces of functions.  Nevertheless, since Feller semigroups are defined in terms of $\CSNOBO$, and this is enough for our purposes, we will stick to this  space.

It has been shown in \cite{Lejay} that the  semigroup of the SNOB is given by:
\begin{theorem}[\cite{Lejay}]\label{lejay}
The semigroup $(\PSNOB)_{t\geq 0}: \CSNOBO\to\CSNOBO$ of the SNOB with parameter $\kappa$ is given by
\begin{equation}\label{formula}
\begin{split}
\PSNOB f(u)\;=\; & \bb E_u\Big[\Big(\frac{1+e^{-\kappa L(0,t)}}{2}\Big)\,f\big(\sgn(u)\vert B_t\vert\big)\Big]\\
& + \bb E_u\Big[\Big(\frac{1-e^{-\kappa L(0,t)}}{2}\Big)\,f\big(-\sgn(u)\vert B_t\vert\big)\Big]\,, \qquad \forall\, u\in \bb G\,,
\end{split}
\end{equation}
where  $(B_t)_{t\geq 0}$ is a standard Brownian Motion starting from $u\neq 0$ and $L(0,t)$ is its local time  at zero.
\end{theorem}

Above, it is understood that $\sgn(u)=1$ if $u\in [0^+,\infty)$ and $\sgn(u)=-1$ if 
$u\in (-\infty,0^-]$. For the sake of clarity, let us  briefly review the notion of local time for the BM. The \textit{occupation measure} of $(B_t)_{t\geq 0}$ up to time instant $t$ is the (random) measure $\mu_t$ defined by the equality
\begin{equation*}
\mu_t(A) \;=\; \int_0^t \mathds{1}_A(B_s)\,ds\,,\qquad \forall \, A\in \mathcal{B}\,,
\end{equation*}
where $\mathds{1}_A$ is the indicator function of the set $A$, and $\mathcal{B}$ are the Borelian sets of $\bb R$. In \cite{Levy39,Levy65}, L\'evy showed that, for almost all trajectories of the BM,  the measure $\mu_t$ has a density $L(u,t)$ with respect to the Lebesgue measure, that is 
\begin{equation*}
\mu_t(A)\;=\; \int_A L(u,t)\,du\,,\qquad\forall\, t\geq 0\,.
\end{equation*}
In \cite{Trotter}, before the advent of stochastic calculus and based on a profound study of the structure of zeros of BM, Trotter proved that there exists a modification of the local time $L(u,t)$ which is continuous on $\bb R \times [0,\infty)$. With a slight abuse of notation, we denote such a modification also by $L(u,t)$. It therefore holds with probability one  that
\begin{equation*}
L(u,t)\;=\; \lim_{\eps \searrow 0} \frac{1}{2\eps}\int_0^t \mathds{1}_{(x-\eps,x+\eps)} (B_s)\,ds\,,\qquad \forall \, (x,t)\in \bb R \times [0,\infty).
\end{equation*}
An equivalent and elegant definition of Brownian local times by means of It\^{o}-calculus is provided by Tanaka's formula
\begin{equation}\label{eq:Tanaka}
L(u,t) \;=\; |B_t-u| - |B_0-u|- \int_0^t \sgn(B_s)\, dB_s\,,
\end{equation} 
which holds for any $u\in \R$, see  for instance \cite[p. 239]{RevuzYor}.
On the equivalence between these two notions of local time, see \cite[p. 224, Corollary~1.6]{RevuzYor}.
 For some history on the development of local times and earlier references, see the survey~\cite{Borodin}, and for a more modern proof on the existence of the jointly continuous modification of the local time, see \cite[p. 225, Theorem~1.7]{RevuzYor}.   

We comment that  \cite{Lejay} only enunciates that the SNOB is a strong Markov process. But the fact that the SNOB is a Feller process is   a simple consequence of  formula \eqref{formula},   continuity, and positiveness of $L(u,t)$,  which put together imply  that $\PSNOB\CSNOBO\subset \CSNOBO$ by    the Dominated Convergence Theorem. 

The main result of this paper consists of the following Donsker-type theorem, which surprisingly connects the \textit{slow bond random walk} with the \textit{snapping out Brownian motion}.
\begin{theorem}\label{thm21}
Let $u\in \bb R\backslash \{0\}$ and consider  the slow bond random walk\break  $\{n^{-1}\XSB_{tn^2}:t\in[0,1]\}$ starting from the site $\lfloor un\rfloor\in \bb Z$. Then, 
$\{n^{-1}\XSB_{tn^2}:t\in[0,1]\}$ converges in distribution, with respect to the $J_1$-topology of Skorohod  of $\mathscr{D}([0,1], \bb R)$, to a process $Y= \{Y_t:t\in[0,1]\}$, where $Y$ is:\medskip

$\bullet$ for $\beta\in [0,1)$,    the  Brownian motion $B$ starting from $u$.\medskip

$\bullet$ for $\beta=1$,   the snapping out Brownian motion $\BSNOB$ of parameter $\kappa=2\alpha$ starting from $u$.\medskip

$\bullet$ for $\beta\in(1,\infty]$,    the reflected Brownian motion $B^{\text{\rm ref}}$ starting from $u$.\smallskip
\end{theorem}
 Above, it is understood that $B^{\text{ref}}$ is the reflected Brownian motion with state space $\bb G$. 
 The semigroup  of $B$ is as is well known
  \begin{equation}\label{semiBM}
  P_t f(u)\;=\; \E_u\big[f(B_t)\big]\;=\; \frac{1}{\sqrt{2\pi t}}\int_{\bb R} e^{-\frac{(u-y)^2}{2t}}f(y)\,dy\,,\quad \textrm{for any }u\in\bb R\,,
 \end{equation}
 while the semigroup of the reflected Brownian motion is given by
   \begin{equation*}
 \PREF  f(u)\;=\;
  \begin{cases}
\displaystyle   \frac{1}{\sqrt{2\pi t}}\int_{0}^{+\infty}\Big[
e^{-\frac{(u-y)^2}{2t}}+e^{-\frac{(u+y)^2}{2t}}\Big]\,f(y)\,dy\,,\quad &\textrm{for }u\in[0^+,\infty)\,,\vspace{0.2cm}\\
\displaystyle  \frac{1}{\sqrt{2\pi t}}\int_{0}^{+\infty}\Big[
e^{-\frac{(u-y)^2}{2t}}+e^{-\frac{(u+y)^2}{2t}}\Big]\,f(-y)\,dy\,,\quad &\textrm{for }u\in(-\infty,0^-]\,.\\
\end{cases}
 \end{equation*}
Next, we  connect the SNOB with a partial differential equation with Robin boundary conditions. 
\begin{proposition}\label{prop23}
Let $(\PSNOB)_{t\geq 0}: \CSNOBO \to \CSNOBO$ be the semigroup of the\break SNOB with parameter $\kappa$. Then, for any  $f\in \CSNOBO$, we have that $\PSNOB f (u)$ is the solution of the partial differential equation 
\begin{equation}\label{pdeRobin}
\begin{cases}
\partial_t \rho(t,u)\;=\;\frac{1}{2}\Delta \rho(t,u)\,, &u\neq 0\\ 
\partial_u \rho(t,0^{+})\;=\;\partial_u\rho (t,0^{-})\;=\;\dfrac{\kappa}{2}\big[\rho(t,0^{+})-\rho(t,0^{-})\big]\,, & t>0\\
\rho(0,u)\;=\;f(u)\,,& u\in \bb R.
\end{cases}
\end{equation}
Moreover, the semigroup $(\PSNOB)_{t\geq 0}: \CSNOBO \to \CSNOBO$ is given by  \begin{equation*}
  \begin{split}
  & \PSNOB f(u)= \frac{1}{\sqrt{2\pi t}}\Bigg\{\int_{\bb R}
e^{-\frac{(u-y)^2}{2t}} \fe (y)\,dy \\
    & + e^{\kappa u}\int_u^{+\infty} e^{-\kappa z} \int_0^{+\infty}
\Big[(\pfrac{z-y+\kappa t}{2t})e^{-\frac{(z-y)^2}{2t}}+(\pfrac{z+y-\kappa t}{2t})e^{-\frac{(z+y)^2}{2t}}\Big]\,
\fo (y)\, dy\, dz\Bigg\},\\
  \end{split}
  \end{equation*}
\noindent for $u>0$ and
  \begin{equation*}
  \begin{split}
 & \PSNOB f(u)= \frac{1}{\sqrt{2\pi t}}\Bigg\{\int_{\bb R}
e^{-\frac{(u-y)^2}{2t}} \fe (y)\,dy \\
    & - e^{-\kappa u}\int_{-u}^{+\infty} e^{-\kappa z} \int_0^{+\infty}
\Big[(\pfrac{z-y+\kappa t}{2t})e^{-\frac{(z-y)^2}{2t}}+(\pfrac{z+y-\kappa t}{2t})e^{-\frac{(z+y)^2}{2t}}\Big]\,\fo (y)\, dy\, dz\Bigg\},\\
  \end{split}
  \end{equation*}
\noindent for $u<0$, where $\fe$ and $\fo$ are the even and odd parts of $f$, respectively.
\end{proposition}

In order to state the Berry-Esseen estimates, we review some further concepts of weak convergence on probability spaces. Given a metric space $(S,d)$, the space of bounded Lipschitz functions $\text{BL}(S)$ is the set of real functions on $S$ such that
\begin{align}
&\Vert f\Vert_\infty \;=\; \sup_{u\in S} |f(u)|\;<\;\infty\,,\quad\text{and}\label{25}\\
&\Vert f\Vert_{\text{L}} \;=\; \sup_{\topo{u,v\in S}{u\neq v}} \frac{|f(u)-f(u)|}{d(u,v)}\;<\;\infty\,.\label{26}
\end{align} 
$\text{BL}(S)$ is a normed linear space with the norm $\Vert f\Vert_{\text{BL}}=\Vert f\Vert_\infty+\Vert f\Vert_{\text{L}}$. This norm is known as the \textit{bounded Lipschitz norm}. Let $\mathcal{P}(S)$ be the set of probability measures on the measurable space $(S,\mathcal{S})$, where $\mathcal{S}$ are the Borelian sets of $S$. The \textit{dual bounded Lipschitz metric} $d_{\text{BL}}$ on $\mathcal{P}(S)$ is defined through
\begin{align}\label{dbl}
d_{\text{BL}}(\mu,\nu)\;=\; \sup_{\topo{f\in \text{BL}(S)}{\Vert f\Vert_{\text{BL}}\leq 1}}
\Big|\int fd\mu-\int fd\nu\,\Big|\,.
\end{align}
Under the additional condition that $(S,d)$ is separable, $d_{\text{BL}}$ becomes a metric for the weak convergence. That is, given $\mu,\mu_n\in\mathcal{P}(S)$, we have that $\mu_n\Rightarrow \mu$ if, and only if, $d_{\text{BL}}(\mu_n,\mu)\to 0$, see \cite[p. 11, Corollary 2.5]{araujoguine} for instance. \smallskip

In this paper, the metric space $\mathcal{S}$ above will be $\bb R$ or $\bb G$. 
The metric space $\bb G= (-\infty, 0^-]\cup [0^+,\infty)$ has two isolated connected components. In such a case, the supremum in \eqref{26} can be restricted to the pairs $x,y$ belonging to the same connected component with no prejudice to the facts above. This will be assumed henceforth. Moreover, the set $\frac{1}{n}\bb Z$ can be embedded into both sets $\bb R$ and $\bb G$. When embedding $\frac{1}{n}\bb Z$ into $\bb G$, one must only have the caution of assuming that $\frac{0}{n}=0^+$ and to look at test functions $f:\bb R\backslash \{0\}\to \bb R$ that are continuous from the right at zero.

\begin{theorem}[Berry-Esseen estimates]\label{thmBE} Fix $t>0$ and $u\neq 0$. Denote by $\mu_{tn^2}^{\text{\rm slow}}$  the probability measure on $\bb R$  induced by the slow bond random walk $\XSB_{tn^2}/n$ starting from $\lfloor un\rfloor$.  Denote by $\mu^{\text{\rm snob}}_t$ and $\mu^{\text{ref}}_t$  the probability measures on $S=\bb G$ induced by $\BSNOB_t$ and $\BREF_t$, respectively, and  denote by $\mu_t$ the probability measure on $S=\bb R$ induced by the Brownian motion $B_t$. All the previous Brownian motions are assumed to start from $u$. We have that: 
\begin{itemize}
\item If $\beta\in [0,1)$, then
\begin{align*}
d_{\text{\rm BL}}(\mu_{tn^2}^{\text{\rm slow}},\mu_t)\;\lesssim \; n^{\beta-1}\,.
\end{align*}
\item If $\beta=1$, then for any $\delta>0$,
\begin{align*}
d_{\text{\rm BL}}(\mu_{tn^2}^{\text{\rm slow}},\mu^{\text{\rm snob}}_t)\;\lesssim \; n^{-1/2 + \delta}\,.
\end{align*}
\item If $\beta\in (1,\infty]$, then
\begin{align*}
d_{\text{\rm BL}}(\mu_{tn^2}^{\text{\rm slow}},\mu^{\text{\rm ref}}_t)\;\lesssim \; \max\{n^{-1},n^{1-\beta}\}\,.
\end{align*}
\end{itemize}
\end{theorem}
We comment that the  convergences above are  slower than the Berry-Essen rate of convergence for the symmetric random walk, which is of order $n^{-1}$ (keep in mind that we are considering the  diffusive time scaling $n^2$). An intuition of why this is so is as follows.  

If $\beta\in[0,1)$, the slow bond random walk converges to the usual Brownian motion. However, the slow bond hinders the passage through the origin, thus making the speed of convergence slower.

If $\beta=1$, as we shall see, an Invariance Principle for local times of the reflected random walk plays a protagonist role in the proof of the result above. It is known that invariance principles for local times of the Brownian motion have speed of convergence\footnote{With respect to the diffusive scaling $n^2$. In the ballistic scaling $n$, used by many authors as \cite{Revesz}, it of course corresponds to a rate  of order $n^{-1/4}$.} of order  at most $n^{-\frac{1}{2}}$. This slower rate of convergence for local times is thus inherited by the rate of convergence for the slow bond random walk.

If $\beta\in(1,\infty]$, the convergence of the slow bond random walk is towards the reflected Brownian motion. In this case,  the slow bond random walk may occasionally jump over the slow bond,  being trapped with high probability in the ``wrong'' half line. This fact  is responsible for a slower rate of convergence. Note that when $\beta\geq 2$, then $\max\{n^{-1},n^{1-\beta}\}=n^{-1}$ and  the slow bond does not interfere  in the rate of convergence. 

\begin{remark}\rm
For the case $\beta=1$, in view of \cite{OnBest} it is natural to expect that the sharpest estimate should be $n^{-1/2}$ times a logarithmic correction. We expect that it would be possible with our methods to obtain such a bound upon analysing carefully and improving existing results on approximations of Brownian local times by random walk local times, for that see in particular Lemma~\ref{lem:RWlocaltimeapprox}, which is a key ingredient.
\end{remark}

\section{An expression for the SNOB semigroup}\label{Sec:semigroup}
Here we prove Proposition~\ref{prop23},  that is, we  show that  the SNOB semigroup is a solution of a heat equation with boundary condition of third (or Robin) type and, moreover, we provide an explicit formula for it. In spite of the obvious importance of having an explicit formula for the semigroup (concerning applications), we explain that its deduction, as we will see, is simply a suitable connection of results from \cite{Lejay} and \cite{fgn2}. Later, this result will be needed in the proof of the central limit theorem for the slow bond random walk.

Denote by $(G_\lambda)_{\lambda>0}$ the resolvent family of the SNOB, which  acts on $f\in \CSNOBO$ via  $G_\lambda f(u)=\E_u\Big[\int_0^\infty e^{-\lambda t} f(\BSNOB_t)\Big]dt = \int_0^\infty e^{-\lambda t} \PSNOB f(u)dt$.
 We recall the following result from \cite{Lejay}.
\begin{proposition}[\cite{Lejay}] For any $f\in \CSNOBO$, the resolvent family $(G_\lambda)_{\lambda>0}$ of the SNOB with parameter $\kappa$ satisfies
\begin{align}
& \Big(\lambda -\frac{1}{2}\Delta\Big) G_\lambda f(u) \;=\; f(u)\,,\quad  u\in \bb G\,, \label{31}\\
& \p_u G_\lambda f(0^+)=\p_u G_\lambda f(0^-)= \frac{\kappa}{2}\big[ G_\lambda f(0^+) - G_\lambda f(0^-)\big]\,.\label{32}
\end{align}
\end{proposition}
Denote by $\mathscr{C}^2_0(\bb G)$ the subspace of twice continuously differentiable functions $f\in \CSNOBO$ such that its first and second derivatives are in $\CSNOBO$. The knowledge on the resolvent family permits to characterize the generator of a Feller process, see \cite[Exercise (1.15) page 290]{RevuzYor} for instance.

  Now denote by $(\PROBIN)_{t\geq 0}: \CSNOBO\to \CSNOBO$ the semigroup determined by  \eqref{pdeRobin}. That is,  $\PROBIN f(u)$ denotes the  solution of the PDE \eqref{pdeRobin} with initial condition $f\in \CSNOBO$. 
  One can easily adapt the result \cite[Proposition 2.3]{fgn2} to deduce that
\begin{equation*}
  \begin{split}
  & \PROBIN f(u)= \frac{1}{\sqrt{2\pi t}}\Bigg\{\int_{\bb R}
e^{-\frac{(u-y)^2}{2t}} \fe (y)\,dy \\
    & + e^{\kappa u}\int_u^{+\infty} e^{-\kappa z} \int_0^{+\infty}
\Big[(\pfrac{z-y+\kappa t}{2t})e^{-\frac{(z-y)^2}{2t}}+(\pfrac{z+y-\kappa t}{2t})e^{-\frac{(z+y)^2}{2t}}\Big]\,
\fo (y)\, dy\, dz\Bigg\},\\
  \end{split}
  \end{equation*}
\noindent for $u>0$ and
  \begin{equation*}
  \begin{split}
 & \PROBIN f(u)= \frac{1}{\sqrt{2\pi t}}\Bigg\{\int_{\bb R}
e^{-\frac{(u-y)^2}{2t}} \fe (y)\,dy \\
    & - e^{-\kappa u}\int_{-u}^{+\infty} e^{-\kappa z} \int_0^{+\infty}
\Big[(\pfrac{z-y+\kappa t}{2t})e^{-\frac{(z-y)^2}{2t}}+(\pfrac{z+y-\kappa t}{2t})e^{-\frac{(z+y)^2}{2t}}\Big]\,\fo (y)\, dy\, dz\Bigg\},\\
  \end{split}
  \end{equation*}
\noindent for $u<0$. A brief resume of this adaptation is given in Appendix~\ref{appendixA} for the sake of completeness.

Thus, in order to conclude the proof of Proposition~\ref{prop23}, it only remains to guarantee that $\PROBIN=\PSNOB$. We claim that  the resolvent family $\GROBIN_\lambda f(u)= \int_0^\infty e^{-\lambda t} \PROBIN f(u)dt$ for \eqref{pdeRobin} also satisfies \eqref{31} and \eqref{32}.
This follows indeed from a direct computation: since $\PROBIN$ is a solution of \eqref{pdeRobin}, we have that
\begin{align*}
 \frac{1}{2}\Delta \GROBIN_\lambda f(u) & \;=\;  \frac{1}{2}\Delta \int_0^\infty\!\! e^{-\lambda t}P^{\text{\tiny\rm Robin}}_tf(u)dt \;=\; 
 \int_0^\infty \!\! e^{-\lambda t} \frac{1}{2}\Delta P^{\text{\tiny\rm Robin}}_tf(u)dt \\
 &\;=\; \int_0^\infty \!\! e^{-\lambda t} \p_t P^{\text{\tiny\rm Robin}}_tf(u)dt\;=\;
 \lambda \int_0^\infty \!\! e^{-\lambda t}  P^{\text{\tiny\rm Robin}}_tf(u)dt - f(u)\,,
\end{align*}
which gives \eqref{31}, 
and \eqref{32} follows by a similar argument.
 This claim  implies that the semigroups $\PROBIN$ and $\PSNOB$ have the same infinitesimal generator. Hence they are equal, see for instance \cite[page 291, Exercise 1.18]{RevuzYor}. This finishes the proof of the Proposition~\ref{prop23}. 

Recall the definition  of $\Vert \cdot \Vert_{L}$ in \eqref{26}. For later use, we present the following corollary of Proposition~\ref{prop23}.
\begin{corollary}\label{cor:Lipschitz}
Let $f\in \CSNOBO$, and consider the SNOB with parameter $\kappa$. Then, for any $t>0$, we have that  $\PSNOB f\in d_{\text{\rm BL}}(\bb G)$ and  
\begin{align*}
\Vert \PSNOB f\Vert_{\text{\rm BL}} \;\leq\; \Vert f\Vert_\infty\Big[1+2\kappa+3\sqrt{\frac{2}{\pi}}\Big]\,.
\end{align*}
\end{corollary}
\begin{proof}
Proposition~\ref{prop23} allows to differentiate $\PSNOB f(u)$, which  allows to infer by long but elementary calculations that 
\begin{equation*}
\begin{split}
\Vert \p_u \PSNOB f\Vert_\infty &\; \leq\;\Vert f\Vert_\infty\Big[2\kappa+3\sqrt{\frac{2}{\pi}}\Big]\,,
\end{split}
\end{equation*}
implying that 
\begin{align*}
\Vert \PSNOB f\Vert_{\text{L}} \;\leq\; \Vert f\Vert_\infty\Big[2\kappa+3\sqrt{\frac{2}{\pi}}\Big]\,. 
\end{align*}
Noting that $\PSNOB f(u)=\bb E_u\big[f(\BSNOB_t)\big]$ is a contraction semigroup with respect to the supremum norm is enough to finish the proof.
\end{proof}
We remark that the well known H\"older continuity of  Brownian local times (see \cite[Corolary 1.8, page 226]{RevuzYor}) and \eqref{formula} may lead to continuity in space of $\PSNOB$. However, it would not lead to the Lipschitz property  above. This is  reasonable: more smoothness is expected when taking averages, which cannot be deduced from pathwise continuity.

\section{Local times}\label{s3}

In the proof of Theorem \ref{thm21} a joint $L^1$-Invariance Principle for the reflected Brownian motion and its local time (at zero) will be required, as well as some extra results about local times. This is the content of this section.

Recall that the local time of a Brownian motion $B$ at the point $u\in \bb R$ at time $t\geq 0$ is denoted here by $L(u,t)$.
 Denote by  $\{X_t:t\geq 0\}$  the   continuous-time symmetric simple random walk on $\bb Z$  starting from zero with jump rates $\lambda(x,y)=1/2$ if $|x-y|=1$ and zero otherwise, and let $\xi(x,t)=\int_{0}^{t} \one_{\{x\}}(X_s) ds$ be its  local time at $x\in \bb Z$. 

The following result shows that  the pair $(X_t, \xi(0,t))$ is close with high probability to the pair $(B_t, L(0,t))$.
\begin{proposition}[\cite{Csaki2009}, Lemma 5.6, and  \cite{Lawler}, Theorem 3.3.3]\label{lem:RWlocaltimeapprox}
There exists a probability space $(\Omega, \mc F,\bb P)$ such that one can define on it a continuous-time symmetric random walk $X_t$ on $\bb Z$  and a real valued Brownian motion  $\big\{B_t:t\geq 0\big\}$ such that there are positive constants $C_1=C_1(t)$ and $C_2=C_2(t)$ such that for any $\delta\in (0,\frac12)$, any $C>0$ any $n\geq 1$ and any $t\geq n^{-2}$ we have the estimate
\begin{equation}\label{eq:rwlocaltimeapprox}
\bb P\Big[\big| \xi(0,tn^2) -L(0,tn^2)\big| \geq 2t^{\frac14 +\delta}n^{\frac{1}{2}+2\delta} + C\log n\Big] \leq C_1\big(n^{\frac{1}{2} -\frac{\delta}{2}} e^{-C_2 n^{\delta}} + n^{1 +\delta -C}\big).
\end{equation}
Moreover, for the same coupling  there are constants $0 < c,a < \infty$ such that, for any $\delta \in (0,1/2]$ and any pair $(t,n)$ as above,
\begin{equation}\label{eq:rwapprox}
\bb P\Big[ \,\sup_{s\leq t} |X_{sn^2}-B_{sn^2}| \geq n^{\frac{1}{2}}\, \Big] \;\leq\; 
ce^{-an^\delta}\,.
\end{equation}
\end{proposition}
We note that \eqref{eq:rwlocaltimeapprox} was originally stated in \cite[Lemma 5.6]{Csaki2009} for the discrete time random walk. In order to translate it into the continuous setting one can apply standard large deviations arguments for the number of jumps and holding times of the continuous time random walk. Using Proposition~\ref{lem:RWlocaltimeapprox} above we deduce the following result.
\begin{proposition}
\label{prop:L1est}
There exists a probability space $(\Omega, \mc F,\bb P)$ such that one can define on it a continuous-time symmetric random walk $X_t$ on $\bb Z$  and a Brownian motion  $\big\{B_t:t\geq 0\big\}$ for which there is a constant $C>0$ such that, for any $\delta >0$, any $n\geq 1$ and any $t\geq n^{-2}$,
\begin{align}
&\E \Big[\, \Big|\frac{\xi(0,tn^2)}{n} -\frac{L(0,tn^2)}{n}\Big|\,\Big] \;\leq\; Cn^{-1/2 +\delta}\label{eq:localtimeL1}, \quad\text{and}\\
&\E\Big[ \frac1n\big|X_{tn^2}-B_{tn^2}\big|\Big] \;\leq\; Cn^{-1/2+\delta}.\label{eq:rwL1}
\end{align}
\end{proposition}
\begin{proof}
We only prove \eqref{eq:localtimeL1} since the proof of \eqref{eq:rwL1} follows the same lines of reasoning.	
We use the abbreviation 
\begin{equation*}
A_n \;=\; \frac{\xi(0,tn^2)}{n} -\frac{L(0,tn^2)}{n}\,.
\end{equation*}
Let $\delta\in (0,1/2)$. We now write
\begin{equation*}
\begin{aligned}
\E \big[\, |A_n|\,\big] 
&\;=\; \E\big[\, |A_n|\, \mathds{1}_{\{|A_n|\leq 3t^{\frac14 +\delta}n^{-\frac{1}{2}+2\delta}\}}\,\big] +
 +\E \big[\,|A_n|\, \mathds{1}_{\{|A_n|>3t^{\frac14 +\delta} n^{-\frac{1}{2}+2\delta}\}}\,
\big] \,.
\end{aligned}
\end{equation*}
The first term on the right hand side of above is bounded by  $3t^{\frac14 +\delta}n^{-\frac{1}{2}+2\delta}$. To bound the second term we apply the Cauchy-Schwarz inequality to see that
\begin{equation*}
\E \big[\, |A_n|\,\mathds{1}_{\{|A_n|> 3t^{\frac14 +\delta} n^{-\frac{1}{2}+2\delta}\}}\,\big] \;\leq\; \E \big[\,|A_n|^2\,\big]^{\frac{1}{2}} \;\bb P\big[\,|A_n|> 3t^{\frac14 +\delta} n^{-\frac{1}{2}+2\delta}\,\big]^{\frac{1}{2}}.
\end{equation*}
A direct calculation involving the usual local central limit theorem (see for instance \cite[Theorem 2.5.6]{Lawler}) shows that the $L^2$-norm of $\xi(0,tn^2)/n$ is bounded in $n$ (the interested reader may easily adapt the proof of Proposition~\ref{prop:localtimecontinuity} to that end). 

To assure that the same $L^2$-boundedness holds true for $L(0,tn^2)/n$, it is sufficient to  note that the laws of $L(0,tn^2)/n$ and $L(0,t)$ are identical, and  then to apply It\^{o}'s isometry. Recalling  Proposition \ref{lem:RWlocaltimeapprox}  concludes the proof.
\end{proof}

The next step is to adapt  the result above to the context of the reflected random walk  and the reflected Brownian motion. For an illustration of the (continuous-time) reflected random walk $\{\XREF: t\geq 0 \}$, see Figure~\ref{fig1}.
\begin{figure}[!htb]
\centering
\begin{tikzpicture}
\centerarc[thick,->](3.5,-0.3)(-10:-170:0.45);
\centerarc[thick,<-](3.5,0.3)(10:170:0.45);
\centerarc[thick,->](4.5,-0.3)(-10:-170:0.45);
\centerarc[thick,<-](4.5,0.3)(10:170:0.45);
\centerarc[thick,->](5.5,-0.3)(-10:-170:0.45);
\centerarc[thick,<-](5.5,0.3)(10:170:0.45);
\centerarc[thick,->](6.5,-0.3)(-10:-170:0.45);
\centerarc[thick,<-](6.5,0.3)(10:170:0.45);

\draw (3,0) -- (10,0);

\shade[ball color=black](6,0) circle (0.25);

\filldraw[fill=white, draw=black]
(3,0) circle (.25)
(4,0) circle (.25)
(5,0) circle (.25)
(7,0) circle (.25)
(8,0) circle (.25)
(9,0) circle (.25)
;

\draw (3.3,-0.05) node[anchor=north] {\small $\bf 0$};
\draw (4.3,-0.05) node[anchor=north] {\small $\bf 1$};
\draw (5.3,-0.05) node[anchor=north] {\small $\bf 2$};
\draw (6.3,-0.05) node[anchor=north] {\small $\bf 3$};
\draw (7.3,-0.05) node[anchor=north] {\small $\bf 4$};
\draw (3.5,0.8) node[anchor=south]{$\frac{1}{2}$};
\draw (4.5,-0.8) node[anchor=north]{$\frac{1}{2}$};
\draw (4.5,0.8) node[anchor=south]{$\frac{1}{2}$};
\draw (5.5,-0.8) node[anchor=north]{$\frac{1}{2}$};
\draw (5.5,0.8) node[anchor=south]{$\frac{1}{2}$};
\draw (3.5,-0.8) node[anchor=north]{$\frac{1}{2}$};
\draw (6.5,-0.8) node[anchor=north]{$\frac{1}{2}$};
\draw (6.5,0.8) node[anchor=south]{$\frac{1}{2}$};
\end{tikzpicture}
\caption{Reflected random walk on $\{0,1,2,\ldots\}$. All  jump rates are equal to one half.}\label{fig1}
\end{figure}
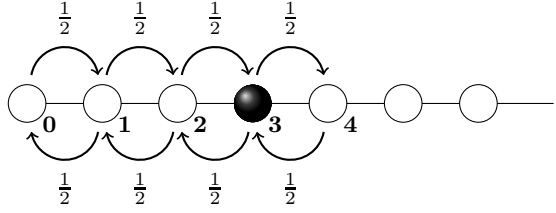

We recall below the notion of  \textit{projection} for continuous-time  Markov chains, also called \textit{lumping} in the literature.
\begin{proposition}[\cite{efgnt}]\label{prop:lumping}
Let $\mc E$ be a countable set, and consider a bounded function $\zeta:\mc E\times \mc E\to[0,\infty)$. Let $(\ZZ_t)_{t\geq 0}$ be the continuous time Markov chain with state space $\mc E$  and jump rates $\{\zeta(x,y)\}_{x,y\in \Omega}$. 
Fix an equivalence relation $\sim $ on $\mc E$ with equivalence classes $\mc E^{\sharp}=\{[x]:\, x\in \mc E\}$ and assume that, for any $y\in \mc E$, 
\begin{align}\label{eq46}
\sum_{y'\sim y} \zeta (x,y')\;=\; \sum_{y'\sim y} \zeta (x', y')
\end{align}
whenever $x\sim x'$. Then, $\big([\ZZ_t]\big)_{t\geq 0}$ is a Markov chain with state space $\mc E^{\sharp}$  and jump rates $\zeta ([x], [y])= \sum_{y'\sim y} \zeta (x,y')$.
\end{proposition}
Consider now the following equivalence relation on $\bb Z$. We will say that $x\sim y$ if, and only if,
\begin{equation*}
x\;=\;y\quad \text{ or } \quad x\;=\; -y-1\,.
\end{equation*}
The equivalence classes of $\bb Z/\!\!\!\! \sim$ will be  therefore $\{-1,0\},\{-2,1\},\{-3,2\},\ldots$ Then, assuming that $\ZZ_t$ is the   continuous-time symmetric slow bond random walk $\XSB_t$ on $\Z$, Proposition~\ref{prop:lumping} tell us that  the projected Markov chain $[\XSB_t]$ has the rates of the reflected random walk $\XREF$, see Figure~\ref{fig2}.
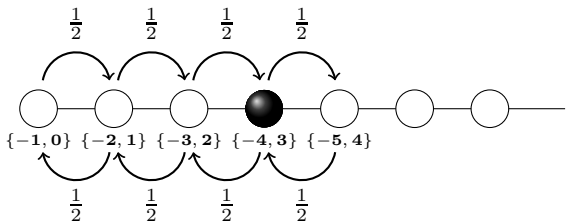
\begin{figure}[H]
\centering
\begin{tikzpicture}
\centerarc[thick,->](3.5,-0.5)(-10:-170:0.45);
\centerarc[thick,<-](3.5,0.3)(10:170:0.45);
\centerarc[thick,->](4.5,-0.5)(-10:-170:0.45);
\centerarc[thick,<-](4.5,0.3)(10:170:0.45);
\centerarc[thick,->](5.5,-0.5)(-10:-170:0.45);
\centerarc[thick,<-](5.5,0.3)(10:170:0.45);
\centerarc[thick,->](6.5,-0.5)(-10:-170:0.45);
\centerarc[thick,<-](6.5,0.3)(10:170:0.45);

\draw (3,0) -- (10,0);

\shade[ball color=black](6,0) circle (0.25);

\filldraw[fill=white, draw=black]
(3,0) circle (.25)
(4,0) circle (.25)
(5,0) circle (.25)
(7,0) circle (.25)
(8,0) circle (.25)
(9,0) circle (.25)
;

\draw (3,-0.2) node[anchor=north] {\tiny $\{\bf{-1},\bf 0\}$};
\draw (4,-0.2) node[anchor=north] {\tiny $\{\bf{-2},\bf 1\}$};
\draw (5,-0.2) node[anchor=north] {\tiny $\{\bf{-3},\bf 2\}$};
\draw (6,-0.2)  node[anchor=north] {\tiny $\{\bf{-4},\bf 3\}$};
\draw (7,-0.2)  node[anchor=north]{\tiny $\{\bf{-5},\bf 4\}$};
\draw (3.5,0.8) node[anchor=south]{$\frac{1}{2}$};
\draw (4.5,-0.95) node[anchor=north]{$\frac{1}{2}$};
\draw (4.5,0.8) node[anchor=south]{$\frac{1}{2}$};
\draw (5.5,-0.95) node[anchor=north]{$\frac{1}{2}$};
\draw (5.5,0.8) node[anchor=south]{$\frac{1}{2}$};
\draw (3.5,-0.95) node[anchor=north]{$\frac{1}{2}$};
\draw (6.5,-0.95) node[anchor=north]{$\frac{1}{2}$};
\draw (6.5,0.8) node[anchor=south]{$\frac{1}{2}$};
\end{tikzpicture}
\caption{Projected Markov chain $[\XSB_t]$ on the state space $\Omega=\bb Z/\!\!\sim$\,. All  jump rates are equal to one half.}\label{fig2}
\end{figure}
Therefore, based on the construction above, we deduce that the local time at zero of the reflected random walk is almost surely equal  to local time of the usual random walk on the set $\{-1,0\}$ (in this coupling). 
\begin{remark}\rm
Note that the usual symmetric continuous time random walk on $\bb Z$ is a particular case of $\XSB_t$ taking $\beta=0$.
\end{remark}
\begin{remark}
\rm In the discrete time setting, it is true that the modulus of the symmetric random walk is the reflected random walk. However, the same does not hold in the continuous time setting, due to the fact that the waiting time at zero would be doubled when taking the modulus. This explains the choice of  the equivalence relation above, which uses symmetry around the point~$-1/2$.
\end{remark}

The next result is quite intuitive, but not so immediate to prove: the times spent by the usual random walk at  sites  $-1$ and $0$ are very close.
\begin{proposition}\label{prop:localtimecontinuity}
Uniformly on $x\in \bb Z$, we have the estimate
\begin{equation}\label{l2}
\E_x\Big [\,\Big(\frac{\xi(0,tn^2)}{n} - \frac{\xi(-1,tn^2)}{n}\Big)^2\,\Big] \;\lesssim\; \frac{1}{n}\,.
\end{equation}
In particular,
\begin{equation}\label{l1}
\E_x\Big [\,\Big|\frac{\xi(0,tn^2)}{n} - \frac{\xi(-1,tn^2)}{n}\Big|\,\Big] \;\lesssim\; \frac{1}{\sqrt{n}}\,.
\end{equation}
\end{proposition}
\begin{proof}  First of all,  observe that the function
\begin{align*}
f(x)\;=\;\E_x\Big [\,\Big(\frac{\xi(0,tn^2)}{n} - \frac{\xi(-1,tn^2)}{n}\Big)^2\,\Big]
\end{align*}
is such that $f(x)\leq f(0)=f(1)$ for any $x\in \bb Z$. The reason is simple: while the random walk does not reach $0$ nor $-1$, both local times above stay null, which gives the inequality, while the equality  is due to symmetry. Hence, let us assume without loss of generality that $x=0$. Applying  the definition of the local time, a change of variables and symmetry, we obtain that
\begin{align*}
& \E_0\Big [\,\Big(\frac{\xi(0,tn^2)}{n} - \frac{\xi(-1,tn^2)}{n}\Big)^2\,\Big]=  n^2 \E_0\bigg[\Big(\!\int_{0}^{t}\!\!\!\big(\mathds{1}\{X_{sn^2}\!=\!-1\}-\mathds{1}\{X_{sn^2}\!=\!0\}\big)ds\Big)^2\bigg]\\
&=\; 2n^2\,\E_0\bigg[\int_{0}^{t}\!\!ds_1\int_{0}^{s_1}\!\!ds_2\bigg(\mathds{1}\{X_{s_1n^2}=X_{s_2n^2}=-1\}-\mathds{1}\{X_{s_1n^2}=0\,,\,X_{s_2n^2=-1}\} \\ 
&-\mathds{1}\{X_{s_1n^2}=-1\,,\,X_{s_2n^2=0}\}+\mathds{1}\{X_{s_1n^2}=X_{s_2n^2}=0\}\bigg)\bigg]\,.
\end{align*}
Interchanging expectation and integrals and applying the Markov property, the above becomes
\begin{align}
&2n^2\int_{0}^{t}ds_1\int_{0}^{s_1}ds_2 \,\Big(\PP_0\big[X_{s_1n^2}=X_{s_2n^2}=-1\big]-\PP_0\big[X_{s_1n^2}=0\,,\,X_{s_2n^2}=-1\big] \notag\\ 
&- \PP_0\big[X_{s_1n^2}=-1\,,\,X_{s_2n^2}=0\big]+\PP_0\big[X_{s_1n^2}=X_{s_2n^2}=0\big]\Big)
\notag\\
&=\; 2n^2\int_{0}^{t}ds_1\int_{0}^{s_1}ds_2\,\Big(\PP_0\big[X_{s_2n^2}=-1]\cdot \PP_{-1}[X_{(s_1-s_2)n^2}=-1\big]\notag \\
&\hspace{3.7cm}-\PP_0[X_{s_2n^2}=-1]\cdot\PP_{-1}\big[X_{(s_1-s_2)n^2}=0\big]\notag\\
&\hspace{3.7cm}-\PP_0\big[X_{s_2n^2}=0]\cdot\PP_{0}[X_{(s_1-s_2)n^2}=-1\big]\notag\\
&\hspace{3.7cm} + \PP_0\big[X_{s_2n^2}=0\big]\cdot \PP_0\big[X_{(s_1-s_2)n^2}=0\big]\Big)\,.\label{eq4.6a}
\end{align}
By symmetry and translation invariance  of the random walk, the integrand above can be rewritten simply as
\begin{align}
&\Big(\PP_0\big[X_{s_2n^2}=-1\big]+\PP_0\big[X_{s_2n^2}=0\big]\Big)\cdot \Big(\PP_0\big[X_{(s_1-s_2)n^2}=0\big]-\PP_0\big[X_{(s_1-s_2)n^2}=1\big]\Big)\notag \\
&=:\; {\bf F}\big(s_2n^2,(s_1-s_2)n^2\big)\;=\;{\bf F} \,.\label{eqsplit}
\end{align}
\begin{figure}[H]
\centering
\begin{tikzpicture}
\begin{scope}[scale=1]
\fill[fill=black!15!white] (0,0)--(5,0)--(5,5)--cycle;
\draw (1,0.5) node{$\bf{B}$};
\draw (4,2) node{$\bf{A}$};
\draw (4,0.5) node{$\bf{C}$};
\draw (4,3.5) node{$\bf{D}$};
\draw (1,0)--(5,4);
\draw (1,1)--(5,1);
\draw (5,0)--(5,5);
\draw (0,0)--(5,5);
\draw[dashed] (0,1)--(1,1);
\draw[->] (-0.5,0)--(5.5,0) node[below]{$s_1$};
\draw[->] (0,-0.5)--(0,5.5) node[left]{$s_2$};
\draw[dashed] (0,5)--(5,5);
\draw (0,1) node[left]{$\frac{2}{n^2}$};
\draw (0,5) node[left]{$t$};
\draw (1,0) node[below]{$\frac{2}{n^2}$};
\draw (5,0) node[below]{$t$};
\end{scope}
\end{tikzpicture}
\caption{Region of integration (in gray) divided into   $\bf{A}$, $\bf{B}$, $\bf{C}$ and $\bf{D}$.}
\label{fig3}
\end{figure}
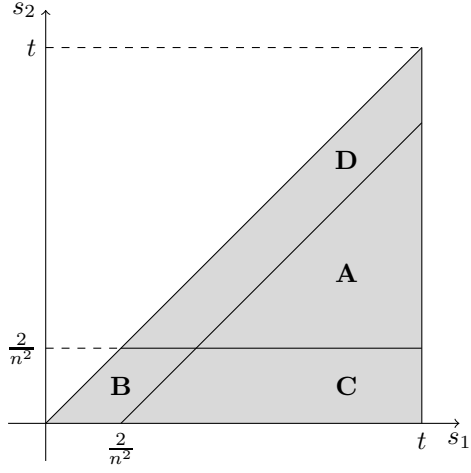
We make now some considerations on how to estimate each factor in \eqref{eqsplit}. Let
\begin{equation*}
p_t(x)\;\overset{\text{def}}{=}\;\PP_0\big[X_{t}=x\big] \qquad \text{and} \qquad
K_t(x)\;\overset{\text{def}}{=}\;\frac{e^{-x^2/2t}}{\sqrt{2\pi t}}\,.
\end{equation*}
By the Local Central Limit Theorem (see \cite[Theorem 2.5.6, p. 66]{Lawler}), it is known that 
\begin{align*}
p_t(x)\;=\; K_t(x) \exp\Big\{O\Big(\frac{1}{\sqrt{t}}+\frac{|x|^3}{t^2}\Big)\Big\}
\end{align*}
in the time range $ t\geq 2|x|$. In particular,
\begin{align}\label{app1}
|p_t(x)| \;\lesssim\; \frac{1}{\sqrt{t}}\quad \text{ for } t\geq 2|x|\,.
\end{align}
Furthermore, adapting to the continuous time setting the result \cite[Theorem 2.3.6, p. 38]{Lawler}, we have also the approximation
\begin{equation}\label{app2}
\big|p_t(x)-p_t(y)- \big(K_t(y)-K_t(x)\big)\big| \;\lesssim \; \frac{|y-x|}{t^{(d+3)/2}}\;=\; \frac{|y-x|}{t^{2}}\,.
\end{equation}
 where $d=1$ is the  dimension in our case. We are going to use this approximation only in the time range $t\geq 2|y-x|$ since for all other values of $t$ it turns out to be not useful for our purposes.
Noting that
\begin{equation*}
|{K}_t(0)-{K}_t(1)|\;\leq\;\sup_{x\in[0,1]}|\partial_u {K}_t(x)|\;=\;\sup_{x\in[0,1]}\Big|\frac{xe^{-x^2/2t}}{t\sqrt{2\pi t}}\Big|\;\lesssim\; \dfrac{1}{t^{3/2}}\,,
\end{equation*}
 we conclude that 
\begin{equation}\label{limitneighbour}
\big|{K}_{(s_1-s_2)n^2}(0)-{K}_{(s_1-s_2)n^2}(1)\big|\;\lesssim\; \dfrac{1}{(s_1-s_2)^{3/2}n^3}\,.
\end{equation}
Since the approximations \eqref{app1} and \eqref{app2} only hold for times $t\geq 2|x|$, we must divide the analysis of \eqref{eq4.6a} in cases, which will be made splitting the region of integration in disjoint sets, as depicted in Figure~\ref{fig3}.\medskip

\noindent \textbf{Region} $\bf{A}$. 
Here $s_2\geq 2/n^2$ and $|s_1-s_2|\geq 2/n^2$.  Restricted to this region, both approximations \eqref{app1} and \eqref{app2} are valid. Recalling \eqref{limitneighbour}, we then get that
\begin{align*}
|{\bf F}|&\;\lesssim\; \Big(\PP_0\big[X_{s_2n^2}=-1\big]+\PP_0\big[X_{s_2n^2}=0\big]\Big)\\
&\times \bigg(\dfrac{1}{((s_1-s_2)n^2)^2}+\Big|K_{(s_1-s_2)n^2}(0)-K_{(s_1-s_2)n^2}(1)\Big|\bigg)\\
&\;\lesssim\; \dfrac{1}{\sqrt{ s_2n^2}}\cdot \bigg(\dfrac{1}{((s_1-s_2)n^2)^2}+\dfrac{1}{(s_1-s_2)^{3/2}n^3}\bigg)\\
&\;\lesssim\; \dfrac{1}{\sqrt{ s_2n^2}}\cdot \dfrac{1}{(s_1-s_2)^{3/2}n^3}\;\lesssim\;\dfrac{1}{\sqrt{ s_2}}\cdot \dfrac{1}{(s_1-s_2)^{3/2}n^4}\,.
\end{align*}
Applying this bound and Fubini's Theorem we obtain that 
\begin{align*}
&2n^2\iint \limits_{\textbf{A}} ds_1\,ds_2\, |{\bf F}|\;\lesssim\; n^2\int_{2/n^2}^{t-2/n^2} ds_2 \int_{s_2+2/n^2}^{t}  ds_1\,\dfrac{1}{\sqrt{s_2}(s_1-s_2)^{3/2}n^4}\\
&=\; \frac{4}{n\sqrt{2}}\Big(\sqrt{t-2/n^2}-2/n\Big)-\frac{2}{n^2}\Big(\arcsin\big(\frac{t-2/n^2}{t}\big)-\arcsin\big(\frac{2}{tn^2}\big)\Big)\;\lesssim \; \frac{1}{n}\,.
\end{align*}

\noindent \textbf{Region} $\bf{B}$. Here $s_2< 2/n^2$ and $|s_1-s_2|< 2/n^2$. Restricted to this region, neither \eqref{app1} nor \eqref{app2} are valid. Nevertheless, since $|{\bf{F}}|\leq 2$,
\begin{align*}
2n^2\iint \limits_{\textbf{B}} ds_1\,ds_2\, |{\bf F}|\;\leq\; 2n^2 \cdot \frac{4}{n^4}\;\lesssim \; \frac{1}{n^2}\,.
\end{align*}

\noindent \textbf{Region} $\bf{C}$. 
Here $s_2< 2/n^2$ and $|s_1-s_2|\geq 2/n^2$, where only the approximation \eqref{app2} is valid. We then have that
\begin{align*}
&2n^2\iint \limits_{\textbf{C}} ds_1\,ds_2 \,|{\bf F}|\;\lesssim\; 2n^2\int_{0}^{2/n^2}ds_2\int_{s_2+2/n^2}^{t}ds_1 \,\dfrac{2}{(s_1-s_2)^{3/2}n^3}\\
& =\; \frac{4}{n}\Big(\sqrt{1-2/n^2}+\frac{4}{n\sqrt{2}}-\frac{16}{n}\Big)\;\lesssim \; \frac{1}{n}\,.
\end{align*}
\noindent \textbf{Region} $\bf{D}$. 
Here $s_2\geq  2/n^2$ and $|s_1-s_2|< 2/n^2$, where only the approximation \eqref{app1} is valid. We then have that
\begin{align*}
&2n^2\iint \limits_{\textbf{D}} ds_1\,ds_2\, |{\bf F}|\;\lesssim\; n^2\int_{2/n^2}^{t} \frac{ds_2}{\sqrt{s_2 n^2}} \int^{s_2+2/n^2}_{s_2}  ds_1\;=\; \frac{4}{n}\Big(\sqrt{t}-\frac{2}{n}\Big)\;\lesssim \; \frac{1}{n}\,.
\end{align*}
Putting together the four estimates above gives us \eqref{l2}. Since  the $L^1$-norm is smaller or equal than the $L^2$-norm for probability spaces, we obtain \eqref{l1}.
\end{proof}

\section{CLT for a fixed time and Berry-Esseen estimates}\label{s5}
We begin by fixing some notation on the space of test functions. 
\begin{definition}
For any $\beta\geq 0$ we define the space $\LSNOB$ via
\begin{equation}
\LSNOB\;=\;
\begin{cases}
\text{\rm BL}(\bb G),\, &\text{if  $\beta\in [1,\infty]$},\\
\text{\rm BL}(\bb R),\, &\text{if $\beta\in [0,1)$.}
\end{cases}
\end{equation} 
\end{definition}
Fix henceforth $f\in \LSNOB$ and denote $\pfrac{1}{n}\bb Z=\{\ldots,-\pfrac{2}{n}, -\pfrac{1}{n}, \pfrac{0}{n}, \pfrac{1}{n},\ldots\}$. Let   $g:[0,\infty)\times \pfrac{1}{n}\bb Z \to\bb R$ be given by
\begin{align}\label{41}
g(t,\pfrac{x}{n})\;=\;
g_t\big(\pfrac{x}{n}\big)\;=\;\E_x\Big[f\Big(\frac{\XSB_{tn^2}}{n}\Big)\Big]\,.
\end{align}
Since the slow bond random walk depends on $n$, so does the function $g$, whose dependence on $n$  has been dropped to not overload notation. Our goal is to prove the CLT directly by studying the convergence of \eqref{41} instead of other traditional methods, as convergence of moments, characteristics functions etc.
The forward Fokker-Planck equation for the generator in \eqref{generator} then yields
\begin{equation}\label{pde15}
\begin{cases}
\partial_t g_t\bigP{x}{n}=\frac{n^2}{2}\big[g_t\bigP{x+1}{n}+g_t\bigP{x-1}{n}-2g_t\bigP{x}{n}\big]\,,&\forall \,x\neq -1,0 \vspace{4pt}\\ 
\partial_t g_t\bigP{0}{n}=\frac{n^2}{2}\big[g_t\bigP{1}{n}-g_t\bigP{0}{n}\big]+\frac{\alpha n^{2-\beta}}{2}\big[g_t\bigP{-1}{n}-g_t\bigP{0}{n}\big],\vspace{4pt}\\
\partial_t g_t\bigP{-1}{n}=\frac{n^2}{2}\big[g_t\bigP{-2}{n}-g_t\bigP{-1}{n}]+\frac{\alpha n^{2-\beta}}{2} [g_t\bigP{0}{n}-g_t\bigP{-1}{n}\big],\vspace{4pt}\\
g\big(0,\pfrac{x}{n}\big)=f(\pfrac{x}{n})\,,&\forall\, x\in \bb Z\,.
\end{cases}
\end{equation}
Note the resemblance of \eqref{pde15} above with the discrete heat equation. To continue we make some symmetry considerations. Let us consider the following  notion of parity for functions  $f:\pfrac{1}{n}\bb Z\to \bb R$, where  the symmetry axis is located at $-\pfrac{1}{2n}$ instead of the origin. That is, we will say that $\feven:\pfrac{1}{n}\bb Z\to \bb R$ is an \textit{even function} if 
\begin{equation}\label{even}
\feven\bigP{x}{n}\;=\;\feven\bigP{-1-x}{n}\,,\quad\forall x\in\Z\,,
\end{equation}
 while by an \textit{odd function} we will mean that 
 \begin{equation}\label{odd}
 \fodd\bigP{x}{n}\;=\;-\fodd\bigP{-1-x}{n}\,,\quad\forall x\in\Z\,.
 \end{equation}
  The even and odd parts of a given function $f:\frac{1}{n}\bb Z\to \bb R$ are hence given by 
\begin{equation*}
\feven\bigP{x}{n}\;=\;\frac{f\bigP{x}{n}+f\bigP{-1-x}{n}}{2}\quad \text{ and } \quad \fodd\bigP{x}{n}\;=\;\frac{f\bigP{x}{n}-f\bigP{-1-x}{n}}{2}
\,,
\end{equation*} 
and it is clear that $f\bigP{x}{n}\;=\;\feven\bigP{x}{n}+\fodd\bigP{x}{n}$. 
  Denote by ${\bf P}^n_tf(x)$ the solution of the semi-discrete scheme \eqref{pde15} with initial condition  $f$.  Due to linearity, 
\begin{equation*}
{\bf P}^n_tf\; =\; {\bf P}^n_t \feven + {\bf P}^n_t \fodd\,.
\end{equation*}
Next, we argue by a simple probabilistic argument that the semi-discrete\break scheme \eqref{pde15} preserves parity, which is an indispensable ingredient in this work. \pagebreak
\begin{proposition}[Parity invariance]\label{prop51}
The semigroup ${\bf P}^n_t$ preserves parity as defined \eqref{even} and \eqref{odd}. That is, if $h:\frac{1}{n}\bb Z\to \bb R$ is  even (respectively, odd), then  ${\bf P}^n_t h$ is even (respectively, odd) for all positive  times.
\end{proposition}
\begin{proof}
 By  symmetry of the jump rates, the distribution of $\XSLOW_{tn^2}$ starting from $x\in \bb Z$ is equal to the distribution of the stochastic process $-1-\XSLOW_{tn^2}$ with $\XSLOW_{tn^2}$ starting from $-1-x$. 
 
Suppose that $h:\frac{1}{n}\bb Z\to \bb R$ is  even, that is, $h\bigP{x}{n}\;=\;h\bigP{-1-x}{n}$. Hence
\begin{align*}
g\big(t,\pfrac{x}{n}\big)&\;=\; \bb E_{x}\Big[h\Big(\frac{\XSLOW_{tn^2}}{n}\Big)\Big]\;=\; 
\bb E_{-1-x}\Big[h\Big(\frac{-1-\XSLOW_{tn^2}}{n}\Big)\Big]\\
&\;=\; \bb E_{-1-x}\Big[h\Big(\frac{\XSLOW_{tn^2}}{n}\Big)\Big]\;=\; g\big(t,\pfrac{-1-x}{n}\big)\,,\qquad\forall\,t>0\,,
\end{align*}
which means that ${\bf P}^n_t h$ is an even function. The argument for an odd function $h$ is analogous.
\end{proof}
Let us discuss the case when \eqref{pde15} starts from $\feven$.  Under our notion of parity, an even function $h$ satisfies $h\bigP{-1}{n}= h\bigP{0}{n}$. This observation together with Proposition~\ref{prop51} allows us to  replace the factors $\alpha n^{2-\beta}/2$ appearing in \eqref{pde15} by any factor. In particular, we may replace those factors by $n^2/2$, thus concluding  that  ${\bf P}^n_t \feven\bigP{x}{n}$ is also a solution of 
\begin{equation}\label{pde17}
\begin{cases}
\partial_t g_t\bigP{x}{n}=\frac{1}{2}\Delta_n g_t\bigP{x}{n}\,, & x\in \bb Z\,,\\ 
g\big(0,\pfrac{x}{n}\big)=\feven\bigP{x}{n}\,,& x\in \bb Z\,,
\end{cases}
\end{equation}
which is the well-known \textit{discrete heat equation}, where $\Delta_n g(x) := n^2 \big[g\bigP{x+1}{n}+g\bigP{x-1}{n}-2g\bigP{x}{n}\big]$ is the discrete Laplacian. Since the discrete heat equation is also the forward Fokker-Planck equation for the symmetric random walk speeded up by $n^2$, we have therefore concluded that
\begin{align}\label{evenPt}
{\bf P}^n_t \feven\bigP{x}{n}\;=\; \E_x\Big[\feven\Big(\frac{X_{tn^2}}{n}\Big)\Big]\,,
\end{align}
where $X_{tn^2}$ is the usual continuous-time symmetric random walk. Of course, now  the classic central limit theorem gives us the desired convergence  towards the expectation with respect to the Brownian motion $B_t$. There is only  one detail to be handled: the notion of parity previously stated was defined on $\pfrac{1}{n}\bb Z$, not on $\bb R$, that is, given $f:\bb R\to \bb R$, the function $\feven:\pfrac{1}{n}\bb Z\to \bb R$ as previously defined depends on the chosen value of $n\in \bb N$. Denote  by $\fe, \fo:\bb R\to\bb R$ the standard even and odd parts of $f$, that is,
\begin{equation*}
\fe(u)\;=\;\frac{f(u)+f(-u)}{2}\quad \text{ and } \quad \fo(u)\;=\;\frac{f(u)-f(-u)}{2}
\,,\quad \forall\,u\in \bb R\,.
\end{equation*} 
It is a simple task to check that
\begin{align}\label{eq47}
\Big|\feven\bigP{x}{n}- \fe\bigP{x}{n}\Big|\;\leq \; \frac{K}{2n} \,,\qquad \forall\, x\in \bb Z\,,
\end{align}
where $K$ is the Lipschitz constant of $f\in \LSNOB$. Recall that $P_t$ is the Brownian semigroup, as defined in \eqref{semiBM}. We have henceforth gathered   the ingredients to deduced the following result:
\begin{lemma}\label{lem:evenapprox}
Let $f\in \LSNOB$. Then, there exists a constant $C>0$ such that for all $t>0$, all $\delta>0$ and all $u\in\R$ we have the estimate 
\begin{equation}\label{eq:evenapprox}
\big|{\bf P}^n_t \feven \bigP{\lfloor un\rfloor}{n} - {P}_t\fe (u)\big|\;\leq\; \frac{C}{n}\,.
\end{equation} 
\end{lemma}
The result above is quite standard. However, since we did not find this exact statement  in the literature, we provide  a short proof of it in Appendix~\ref{appendixA}.\medskip

Let us turn our attention to the odd part. Under our notion of parity, an odd function $h: \frac{1}{n}\bb Z\to \bb R$ satisfies $h\bigP{-1}{n}=-h\bigP{0}{n}$. This together with the parity invariance given in Proposition~\ref{prop51} permits to conclude that ${\bf P}^n_t \fodd\bigP{x}{n}$, for $x\geq 0$, is a solution of
\begin{equation}\label{pde16}
\begin{cases}
\partial_t g_t\bigP{x}{n}=\frac{1}{2}\Delta_n g_t\bigP{x}{n}\,, & x\geq 1\,,\vspace{4pt}\\ 
\partial_t g_t\bigP{0}{n}=\frac{n^2}{2} [g_t\bigP{1}{n}-g_t\bigP{0}{n}]-\alpha n^{2-\beta}g_t\bigP{0}{n}\,,\vspace{4pt}\\
g(0,x)=\fodd(x)\,, & x\geq 1\,,
\end{cases}
\end{equation}
which completely determines ${\bf P}^n_t \fodd$ since it is an odd function for all positive times. Define
\begin{align*}
{L}_n^{\text{\tiny ref}} f\bigP{x}{n}\;=\; 
\begin{cases}
\frac{n^2}{2}\big[f\bigP{x+1}{n}+f\bigP{x-1}{n}-2f\bigP{x}{n}\big],& x\geq 1,\vspace{4pt}\\
\frac{n^2}{2}\big[f\bigP{1}{n}-f\bigP{0}{n}\big],& x=0\,,\\
\end{cases}
\end{align*}
which is the generator of the reflected random walk speeded up by $n^2$. Writing $V_n\bigP{x}{n}=-\alpha n^{2-\beta} \mathds{1}_{\{0\}}(x)$, we can  write \eqref{pde16} in the form
\begin{equation*}
\begin{cases}
\partial_t g_t\bigP{x}{n}\;=\;{L}_n^{\text{\tiny ref}} g_t\bigP{x}{n} + V_n\bigP{x}{n}g\bigP{x}{n}\,,&\qquad x\geq 0\,,\vspace{4pt}\\
g\big(0,\pfrac{x}{n}\big)\;=\;\fodd\bigP{x}{n}\,,&\qquad x\geq 0\,.
\end{cases}
\end{equation*}
 The Feynman-Kac Formula, which can be found for instance in \cite[p. 334, Proposition 7.1]{kl}, yields that
\begin{align*}
{\bf P}^n_t \fodd \bigP{x}{n} & \;=\;\E_x\Big[\fodd\Big(\frac{X_{t n^2}^{\text{\tiny ref}}}{n}\Big) \,\exp\Big\{\int_0^t V_n\Big(\frac{X_{s n^2}^{\text{\tiny ref}}}{n}\Big)ds\Big\}\Big] \notag\\
&\;=\;\E_x\Big[\fodd\Big(\frac{X_{t n^2}^{\text{\tiny ref}}}{n}\Big)\exp\Big\{-\alpha n^{-\beta}\ZREF_{t n^2}(0)\Big\}\Big]\,,
\end{align*}
where
\begin{equation*}
\ZREF_{tn^2}(0)\;=\;n^2\int_0^t\mathds{1}_{\{0\}}(X^{\text{\tiny ref}}_{sn^2})\,ds\;=\; \int_0^{tn^2}\mathds{1}_{\{0\}}(X^{\text{\tiny ref}}_{s})\,ds
\end{equation*}
 is the local time at zero of the reflected random walk up to time $tn^2$.
Using the coupling outlined after Proposition~\ref{prop:lumping} which connects the usual symmetric random walk with the reflected random walk, and the fact that $\fodd$ is an odd function in the sense of \eqref{odd},  we then deduce that
\begin{equation}\label{eq510a}
{\bf P}^n_t \fodd \bigP{x}{n} = \E_x\Big[\fodd\Big(\frac1n\Big[\Big|X_{tn^2}+\frac12\Big| -\frac12\Big] \Big)\exp\Big\{-\frac{\alpha}{n^{\beta}}\xi_{tn^2}\big(\{-1,0\}\big)\Big\}\Big].
\end{equation} 
Let now
\begin{equation*}
{Q}_t \fo(u) \; \overset{\text{def}}{=} \; \E_u\Big[\fo\big(|B_t|\big) \exp\big\{-2\alpha L_t(0)\big\}\Big]\,, \quad \forall\, u\in \bb R\,,
\end{equation*}
where we recall that $B_t$ denotes a standard Brownian motion at time $t$ and $L$ denotes its local time. With all these preparations at hand we can now formulate one of the main results of this section.
\begin{lemma}\label{lem:oddapprox}
Let $f\in \LSNOB$, then  for all $t>0$ and all $u \in \bb R$ with $u>0$ we have the estimates 
\begin{itemize}
	\item If $\beta <1$, then 
\begin{equation*}
	\big|{\bf P}^n_t \fodd  \bigP{\lfloor un \rfloor}{n} - {P}_t\fo (u)\big|\;\lesssim\; n^{\beta-1}\,.
	\end{equation*}	
	
	\item If $\beta=1$, then for all $\delta>0$
	\begin{equation*}
	\big|{\bf P}^n_t \fodd  \bigP{\lfloor un \rfloor}{n} - {Q}_t\fo (u)\big|\;\lesssim\; n^{-\frac12 +\delta}\,.
	\end{equation*}
	\item If $\beta >1$, then
	\begin{equation*}
	\big|{\bf P}^n_t \fodd  \bigP{\lfloor un \rfloor}{n}
	-\E_u\big[\fo (|B_t|)\big]\big|\;\lesssim\; \max\big\{n^{-1}, n^{1-\beta}\big\}\,.
	\end{equation*}
\end{itemize}
Here the proportionality constants above are independent of $t$ and $u$.
\end{lemma}
The proof of this lemma will be given in the next two subsections.
\subsection{Proof of Lemma \texorpdfstring{\ref{lem:oddapprox}}{5.3} for \texorpdfstring{$\beta\in[0,1)$}{beta smaller one}}
Fix $u > 0$ and $f\in \LSNOB=\text{\rm BL}(\bb R)$ and recall \eqref{eq510a}.  Since $f$ is Lipschitz continuous, we can replace ${\bf P}_t^n\fodd\bigP{\lfloor un\rfloor}{n}$ by 
\begin{equation}\label{eq:rwfo}
\E_{\lfloor un\rfloor}\Big[\fo\Big(\frac{|X_{tn^2}|}{n}\Big)\exp\Big\{-\frac{\alpha}{n^{\beta}}\xi_{tn^2}(\{-1,0\})\Big\}\Big]
\end{equation}
paying a price of order $n^{-1}$. 

By the strong Markov property applied at the stopping time $T=\inf\{t\geq 0:\, X_{tn^2}=0\}$, we observe now that
\begin{equation}\label{eq:zero}
\E_{\lfloor un\rfloor}\Big[\fo\Big(\frac{X_{tn^2}}{n} \Big)\mathds{1}_{\{T< tn^2\}} \Big]
\,=\, \E_{\lfloor un\rfloor}\Big[\mathds{1}_{\{T< tn^2\}}\, 
\E_0\Big[\fo\Big(\frac{X_{tn^2-T}}{n} \Big)\Big]\Big]\;=\;0\,,
\end{equation}
where the last equality follows from the facts that $\{X_{tn^2}\}_{t\geq 0}\overset{\text{law}}{=}\{-X_{tn^2}\}_{t\geq 0}$ provided $X_0=0$ and that $\fo$ is an odd function in the usual sense. 

Thus, ``adding'' the null term \eqref{eq:zero} to \eqref{eq:rwfo} and then using that on the event $\{T\geq tn^2\}$ the exponential factor is equal to $1$ (recall we are assuming $u>0$ hence not hitting the site $0$ means not hitting $-1$ as well), we see that \eqref{eq:rwfo} equals 
\begin{equation}\label{eq517}
\E_{\lfloor un\rfloor}\Big[\fo\Big(\frac{ X_{tn^2}}{n}\Big)\Big]
+ \E_{\lfloor un\rfloor}\Big[\fo\Big(\frac{|X_{tn^2}|}{n}\Big)\exp\Big\{-\frac{\alpha}{n^{\beta}}\xi_{tn^2}(\{-1,0\})\Big\}\mathds{1}_{\{T < tn^2\}}\Big].
\end{equation}
The distance between the first parcel above and $\E_u[\fo(B_t)]$ is bounded by some constant times $n^{-1}$, which can be seen exactly as in the proof of\break Lemma~\ref{lem:evenapprox}. Hence, in order  to finish the proof of the
Lemma~\ref{lem:oddapprox} for $\beta<1$
it is sufficient to show that the second term  in \eqref{eq517} converges to zero with the desired order.

 The proof that the second term  in \eqref{eq517} vanishes in the limit will  crucially rely on the next lemma, which  may be interpreted as follows: when starting the usual random walk from $\lfloor un \rfloor$ and looking at a time window  of size $tn^2$, either the local time (at the origin) is zero or either it  is reasonably large. Since $\{T<tn^2\}=\{\xi_{tn^2}(0)>0\}$, the situation where the local time vanishes is excluded in the second parcel of \eqref{eq517}, which means the local time is reasonably large, which in turn yields that the exponential in the second parcel of \eqref{eq517} is reasonably small. This outlines the strategy to be followed in the sequel. Let us first state the lemma mentioned above:
\begin{lemma}\label{lem:smalllocaltime}
Let $\gamma \in (0,1)$ and  $\gamma' \in(\gamma, 1)$. Then, there is a constant $C=C(\gamma, \gamma')> 0$ such that for all $n\in\N$ large enough and all $j<  n^{\gamma' -\gamma}$,
\begin{equation}\label{eq5.18a}
\bb P_{\lfloor un\rfloor}\Big[ jn^\gamma < \xi_{tn^2}(0) \;\leq\; (j+1)n^\gamma\Big]\;\leq\;  Cn^{\gamma -1}\,.
\end{equation}	
\end{lemma}
We defer the proof of the lemma to the end of this section and we show first how it implies that the second term  in \eqref{eq517} converges to zero with the desired order.
Fix $\delta\in (0,1-\beta)$. Using that $\xi_{tn^2}(\{-1,0\})\geq \xi_{tn^2}(0)$, we can then estimate the rightmost term in~\eqref{eq517} by $\I + \II$, where
\begin{equation*}
\begin{split}
\I &\;=\; \E_{\lfloor un\rfloor}\Big[\fo\Big(\frac{|X_{tn^2}|}{n}\Big)\exp\Big\{-\frac{\alpha}{n^{\beta}}\xi_{tn^2}(0)\Big\}\mathds{1}_{\{\xi_{tn^2}(0) > n^{\beta+\delta}\}}\Big] \quad\text{and}\\
\II &\;=\; \E_{\lfloor un\rfloor}\Big[\fo\Big(\frac{|X_{tn^2}|}{n}\Big)\exp\Big\{-\frac{\alpha}{n^{\beta}}\xi_{tn^2}(0)\Big\}\mathds{1}_{\{0 < \xi_{tn^2}(0) \leq n^{\beta+\delta}\}}\Big]\,.
\end{split}
\end{equation*}
It is then straightforward to see that the term $\I$ indeed has the desired behaviour. To see that the same also holds for $\II$, we can estimate using that $\fo$ is bounded:
\begin{equation*}
\begin{split}
\II &\;=\; \sum_{j=0}^{n^{\delta}} \E_{\lfloor un\rfloor}\Big[\fo\Big(\frac{|X_{tn^2}|}{n}\Big)\exp\Big\{-\frac{\alpha}{n^{\beta}}\xi_{tn^2}(0)\Big\}\mathds{1}_{\{jn^{\beta} < \xi_{tn^2}(0) \leq (j+1)n^{\beta}\}}\Big]\\
&\;\lesssim\; \sum_{j=0}^{n^{\delta}}\exp\Big\{-\alpha j\Big\}\,\bb P\Big[jn^{\beta} < \xi_{tn^2}(0) \leq (j+1)n^{\beta}\Big]\,.
\end{split}
\end{equation*}
Applying Lemma~\ref{lem:smalllocaltime} with $\gamma=\beta$ is enough to deduce the claim.
\begin{proof}[Proof of Lemma~\ref{lem:smalllocaltime}]
We first derive the above statement for the local time of a \textit{discrete time} random walk, which we denote  by $(S_n)_{n\in \bb N}$. And denote  its local time until time $n$ of the point $a\in\bb Z$ by $\zeta_n(a)$.
By \cite[Equation (27)]{Takacs}, for any $k\in\bb N$, and any $a\geq 0$ we have the formula 
\begin{equation}\label{eq:formulalocaltime}
\bb P_0\big[\zeta_n(a)\geq k\big] \;=\; \bb P_0\big[S_{n-k+1} \geq a +k-1 \big]
+  \bb P_0\big[S_{n-k+1} > a +k-1 \big]\,.
\end{equation}
Note that by symmetry the same formula applies to the local time at zero provided that $S_0=a$. Denoting as above by $T$ the first hitting time of zero and applying the strong Markov property at time $T$, we therefore see that for any $k\geq 2$,
\begin{equation}\label{eq:localtimesmall}
\bb P_{\lfloor un \rfloor}\big[1\leq \zeta_{n^2}(0)< k\big]
\;=\; \E_{\lfloor un \rfloor }\Big[\mathds{1}_{\{T< n^2\}}\,\bb P_0[1\leq \zeta_{n^2-T}(0)< k]\Big]\,.
\end{equation}
Note that under $\bb P_0$ the local time at zero is always strictly positive. Using  \eqref{eq:formulalocaltime} with $a=0$ we deduce that
\begin{equation*}
\bb P_0\big[1\leq \zeta_{n^2-T}(0)< k\big]\; =\; \bb P_0\big[S_{n^2-T-k+1} < k -1\big]- \bb P_0\big[S_{n^2-T-k+1} > k -1\big]\,.
\end{equation*}
Adding and subtracting  the cumulative distribution function $\Phi$ of the standard normal distribution we can write
\begin{equation}\label{eq:difference}
\begin{split}
&\bb P_0\big[1\leq \zeta_{n^2-T}(0)< k\big]\\
&=\;\bb P_0\big[S_{n^2-T-k+1} < k -1\big]- \bb P_0\big[S_{n^2-T-k+1} > k -1\big]\\
&=\; \bb P_0[S_{n^2-T-k+1} < k -1]+ \bb P_0[S_{n^2-T-k+1} \leq k -1] -1\\
&=\; \I(k) +\II(k)\,,
\end{split}
\end{equation}
where
\begin{align*}
&\I(k) =  2\Phi\Big(\frac{k-1}{\sqrt{n^2-T-k +1}}\Big) -1\qquad \text{and}\\
&\II(k) = \bb P_0[S_{n^2-T-k+1} \leq k-1] +
\bb P_0[S_{n^2-T-k+1} < k-1] - 2\Phi\Big(\frac{k-1}{\sqrt{n^2-T-k +1}}\Big).
\end{align*}
Roughly speaking, we may say that $\I(k)$ and $\II(k)$ are close to one whenever $T$ is close to $n^2$. Which would be bad, since we are aiming to show that $\bb P_0\big[1\leq \zeta_{n^2-T}(0)< k\big]$ is small. 

Therefore,  to get a good upper bound on $\bb P_0\big[1\leq \zeta_{n^2-T}(0)< k\big]$, we need to first show  that the probability that $T$ is close to $n^2$ is small, afterwards it  remains to bound $\I$ and $\II$ for those values of $T$ that are reasonably far away from $n^2$.

Let $\gamma'\in (\gamma,1)$ as in the statement of the lemma. The Hitting Time Theorem (see \cite{HK}) states that
\begin{equation}\label{eq:hittingtime}
\bb P_{\lfloor un \rfloor}\big[T=\ell\big] \;=\; \frac{\lfloor un \rfloor}{\ell}\,\bb P_{\lfloor un\rfloor}\big[S_\ell=0\big]\,.
\end{equation}
Applying the Local Central Limit Theorem~\cite[Theorem 2.3.5]{Lawler} we see that
\begin{equation}\label{eq:Tsmall}
\bb P_{\lfloor un \rfloor}\big[\,T\in (n^2-n^{\gamma'}, n^2)\,\big]\; \lesssim\;
un\sum_{\ell= n^2-n^{\gamma'}}^{n^2}\frac{1}{\ell^{\frac32}} + O\big(n^{\gamma'-2}\big)\,.
\end{equation}
Here, one would actually get an extra factor  $e^{-\frac{u^2n^2}{2\ell}}$ in the sum above. Nevertheless, in the considered range of $\ell$'s,  this factor behaves like a constant, hence it is omitted.
Since $\ell\mapsto \frac{1}{\ell^{3/2}}$ is a decreasing function, we have the following inequality
\begin{equation*}
\sum_{\ell= n^2-n^{\gamma'}}^{n^2}\frac{1}{\ell^{\frac32}}\;\leq\;
\int_{n^2-n^{\gamma'}-1}^{n^2}\frac{dx}{x^{\frac32}}\;=\;\frac{1}{\sqrt{n^2-n^{\gamma'}-1}}-\frac{1}{n}\;=\; O\big(n^{\gamma'-2}\big)
\end{equation*}
from which we can infer that the probability in the left hand side of \eqref{eq:Tsmall} is of order $n^{\gamma' -2}$, which by our choice of $\gamma'$ is smaller than $n^{\gamma-1}$. This provides the first ingredient of the proof, that is, \eqref{eq:localtimesmall} is equal to
\begin{align}
&\E_{\lfloor un \rfloor }\Big[\mathds{1}_{\{T< n^2-n^{\gamma'}\}}\,\bb P_0[1\leq \zeta_{n^2-T}(0)< k]\Big]\notag \\
&+\E_{\lfloor un \rfloor }\Big[\mathds{1}_{\{T\in (n^2-n^{\gamma'}, n^2)\}}\,\bb P_0[1\leq \zeta_{n^2-T}(0)< k]\Big]\notag\\
&\lesssim\;\E_{\lfloor un \rfloor }\Big[\mathds{1}_{\{T< n^2-n^{\gamma'}\}}\,\bb P_0[1\leq \zeta_{n^2-T}(0)< k]\Big]+n^{\gamma-1}\,.\label{eq:4}
\end{align}
 We turn to the analysis of $\I$ and $\II$.
To continue, note that 
\begin{align*}
&\bb P_{\lfloor un \rfloor}\big[\,\zeta_{n^2}(0)\in (jn^\gamma,(j+1)n^\gamma]\,\big]\\
&= \;\bb P_{\lfloor un \rfloor}\big[\zeta_{n^2}(0)\in[1, (j+1)n^\gamma\big]
-\bb P_{\lfloor un \rfloor}\big[\zeta_{n^2}(0)\in [1,jn^\gamma]\,\big]\,,
\end{align*}
which can be written in terms of differences of \eqref{eq:4}. Thus, in order to get the desired bounds we need to estimate  $\I((j+1)n^\gamma)-\I(jn^\gamma)$ and $\II((j+1)n^\gamma)-\II(jn^\gamma)$. Using that $x\mapsto e^{-\frac{x^2}{2}}$ is decreasing in $|x|$, we see that
\begin{align}
&\big|\I((j+1)n^\gamma)-\I(jn^\gamma)\big|\;=\; \frac{1}{\sqrt{2\pi}}\int_{\frac{jn^\gamma}{\sqrt{n^2-T-jn^\gamma}}}^{\frac{(j+1)n^\gamma}{\sqrt{n^2-T-(j+1)n^\gamma}}}\; e^{-\frac{x^2}{2}}\; dx \notag\\
&\lesssim\; \Big(\frac{(j+1)n^\gamma}{\sqrt{n^2-T-(j+1)n^\gamma}}
-\frac{jn^\gamma}{\sqrt{n^2-T-jn^\gamma}}\Big)\exp\Big\{-\frac{j^2n^{2\gamma}}{2(n^2-T-jn^\gamma)}\Big\}\label{eq:2}\\
&\overset{\text{def}}{=}\;\A(j,j+1)\,.\notag
\end{align}
Invoking~\eqref{eq:hittingtime}, noting that $T\geq \lfloor un \rfloor$ if the random walk $S$ starts at $\lfloor un \rfloor$, and once again recalling the Local Central Limit Theorem, we can estimate
\begin{align}
&\E_{\lfloor un\rfloor}\Big[ \one_{\{0\leq  T \leq n^2-n^{\gamma'}\}}\big(\I((j+1)n^\gamma)-\I(jn^\gamma)\big)\Big]\label{eq:0} \\ 
&=\; \E_{\lfloor un\rfloor}\Big[ \one_{\{\lfloor un\rfloor \leq T \leq n^2-n^{\gamma'}\}}\big(\I((j+1)n^\gamma)-\I(jn^\gamma)\big)\Big]\notag \\
&\lesssim\; un \sum_{k=un}^{n^2-n^{\gamma'}}\label{eq:I}
\frac{1}{k^\frac32} \exp\Big\{-\frac{u^2n^2}{2k}\Big\}
\A(j,j+1)\,.
\end{align}
To estimate the rightmost term above, we first note that for all $k$ as above
\begin{equation*}
\exp\Big\{-\frac{j^2n^{2\gamma}}{2(n^2-k-jn^\gamma)}\Big\}\;\leq\; \exp\Big\{-\frac{j^2n^{2\gamma}}{2(n^2-un-jn^\gamma)}\Big\}\,.
\end{equation*}
Writing $k=\frac{k}{n^2}n^2$, factoring out a factor $n^2$ of the two square root terms in \eqref{eq:2}, and making a Riemann sum approximation, it is a long but elementary procedure to see that~\eqref{eq:I} is bounded from above by some constant times
\begin{align}
&un^{\gamma -1}\exp\Big\{-\frac{j^2n^{2\gamma}}{2(n^2-un-jn^\gamma)}\Big\}
\notag\\
&\times\int_{\frac{u}{n}}^{1-n^{\gamma'-2}}
\frac{1}{x^{\frac32}}\exp\Big\{-\frac{u^2}{2x}\Big\}
\Bigg(\frac{(j+1)}{\sqrt{1-x-(j+1)n^{\gamma-2}}}
-\frac{j}{\sqrt{1-x-jn^{\gamma-2}}}\Bigg)\, dx\,.\label{eq:3}
\end{align}
Note that $j < n^{\gamma' - \gamma}$, thus $1-x-(j+1)n^{\gamma -2}$ is always positive in the range of $x$ considered. Keeping this in mind one can check that $u$ times the  integral  in \eqref{eq:3} is uniformly bounded  in $n$ and $u$, therefore \eqref{eq:I} is bounded by a constant times 
\begin{align*}
n^{\gamma -1}\exp\Big\{-\frac{j^2n^{2\gamma}}{2(n^2-un-jn^\gamma)}\Big\}\;\lesssim\;
 n^{\gamma -1}\exp\Big\{-Cj^2n^{2(\gamma-1)}\Big\}\; \leq \; n^{\gamma -1}
\end{align*}
for some constant $C>0$, which finally gives us  the bound on \eqref{eq:0}. We now turn to the bound of $\II$, which is easier than the previous bound for $\I$, since there is no necessity to take differences. \textit{Grosso modo}, we may say that
\begin{align*}
\II(k)\;\lesssim\; \frac{1}{\sqrt{n^2-T-k+1}}
\end{align*}
by the usual Berry-Essen estimate for the random walk, see \cite[p. 137, Theorem 3.4.9]{Durrett} for instance (of course, some knowledge on $T$ is needed to make it precise). Therefore,
\begin{align*}
\E_{\lfloor un\rfloor}\Big[ \one_{\{0\leq  T \leq n^2-n^{\gamma'}\}}\II(jn^\gamma)\Big]&\;=\; 
\E_{\lfloor un\rfloor}\Big[ \one_{\{\lfloor un\rfloor\leq  T \leq n^2-n^{\gamma'}\}}\II(jn^\gamma)\Big]\\
&\;\lesssim\;
\sum_{\ell=\lfloor un\rfloor}^{n^2-n^{\gamma'}}\bb P_{\lfloor un\rfloor}\big[T=\ell\big] \frac{1}{\sqrt{n^2-\ell-jn^{\gamma}+1}}\,.
\end{align*}
Applying the Hitting Time Theorem, the last expression above is equal to
\begin{align*}
\lfloor un\rfloor\sum_{\ell=\lfloor un\rfloor}^{n^2-n^{\gamma'}}\frac{1}{\ell}\bb P_{\lfloor un\rfloor}\big[S_\ell=0\big] \frac{1}{\sqrt{n^2-\ell-jn^{\gamma}+1}}\,.
\end{align*}
By the Local Central Limit Theorem, the above is bounded by a constant times
\begin{align*}
&\lfloor un\rfloor\sum_{\ell=\lfloor un\rfloor}^{n^2-n^{\gamma'}}\frac{1}{\ell^{3/2}}\exp\Big\{-\frac{u^2n^2}{2\ell}\Big\} \frac{1}{\sqrt{n^2-\ell-jn^{\gamma}+1}}\\
&\lesssim\; \frac{u}{n^2}\sum_{\ell=\lfloor un\rfloor}^{n^2-n^{\gamma'}}\frac{1}{(\ell/n^2)^{3/2}}\exp\Big\{-\frac{u^2}{2(\ell/n^2)}\Big\} \frac{1}{\sqrt{n^2(1-\frac{\ell-jn^{\gamma}+1}{n^2})}}\\
&= \frac{1}{n} \times \frac{u}{n^2}\sum_{\ell=\lfloor un\rfloor}^{n^2-n^{\gamma'}}\frac{1}{(\ell/n^2)^{3/2}}\exp\Big\{-\frac{u^2}{2(\ell/n^2)}\Big\} \frac{1}{\sqrt{(1-\frac{\ell-jn^{\gamma}+1}{n^2})}}\,.
\end{align*}
Note now that the second factor above is a Riemann sum approximation similar to \eqref{eq:3}.
Thus, uniformly in $j < n^{\gamma'-\gamma}$,
\begin{equation*}
\E_{\lfloor un\rfloor }\Big[\mathds{1}_{\{T\leq n^2-n^{\gamma'}\}} \II(jn^\gamma)\Big]\;
\lesssim\; n^{-1}\,,
\end{equation*}
 immediately implying that 
\begin{equation*}
E_{\lfloor un\rfloor }\Big[\mathds{1}_{\{T\leq n^2-n^{\gamma'}\}} \big(\II(j+1)n^ {\gamma})-\II(jn^\gamma)\big)\Big]\;
\lesssim\; n^{-1}\,,
\end{equation*} 
from which the result follows for the discrete time random walk.  A standard Poissonisation argument now begets the result for the continuous time case.
\end{proof}

\subsection{Proof of Lemma \texorpdfstring{\ref{lem:oddapprox}}{5.3} for \texorpdfstring{$\beta\in [1,\infty]$}{beta bigger 1}}
\begin{proof}
\textbf{Case $\beta=1$.}
Using that $x\mapsto e^{-x}$ defined on $[0,\infty)$ is bounded by one and Lipschitz continuous with Lipschitz constant one, we can estimate\break  $|{\bf P}^n_t \fodd (x) - {Q}_t\fo  \bigP{x}{n}|\lesssim \I + \II$, where
\begin{equation*}
\begin{aligned}
\I&\;=\; \E_{x,x/n}\Big[ \Big|\fodd\Big(\frac1n\Big[\Big|X_{tn^2}+\frac12\Big| -\frac12\Big] \Big)- \fo\Big(\frac1n\big|B_{tn^2}\big|\Big)\Big|\Big] \quad \text{and}\\
\II&\;=\;  \E_{x,x/n}  \Big[\frac1n\big|\xi_{tn^2}(\{-1,0\})-2L_{tn^2}(0)\big|\Big]\,.
\end{aligned}
\end{equation*}
Here, $\E_{x,x/n}$ denotes the expectation induced by the coupling introduced in Proposition \ref{prop:L1est} of $X$ and $B$.
We first estimate $\I$. To that end denote the Lipschitz constant of $f$ by $L$, and note that for any number $a\in \bb Z$ we have the estimate $\big||a+1/2|-1/2 - |a|\big| \leq 1$. Thus, 
\begin{align*}
\I & \;\leq\;  \frac{L}{n} \times \E_{x,x/n}\Big[ \Big|\Big(\Big|X_{tn^2}+\frac12\Big|\Big) -\frac12 - |B_{tn^2}|\Big|\Big]\\
&\;\leq\; \frac{L}{n} + \frac{L}{n}\times\E_{x,x/n}\Big[ \Big| |X_{tn^2}| - |B_{tn^2}|\Big|\Big]\,.
\end{align*}
It now only remains to apply Proposition \ref{prop:L1est} to deduce the desired estimate for $\I$. To estimate $\II$ we write
\begin{equation*}
\II \;\leq\; \frac1n \E_{x,x/n} \big[ \,|\xi_{tn^2}(\{-1,0\}) - 2 \xi_{tn^2}(0)|\,\big]
+ \frac2n \E_{x,x/n} \big[\,|\xi_{tn^2}(0)- L_{tn^2}(0)|\,\big]\,.
\end{equation*}
The result therefore follows from an application of Propositions \ref{prop:L1est} and \ref{prop:localtimecontinuity}.\medskip

\noindent
\textbf{Case $\beta\in(1,\infty]$.}
We adopt the abbreviation 
\begin{equation*}
\frac1n\Big[\Big|X_{tn^2}+\frac12\Big| -\frac12\Big]\;=\; |X_{tn^2}|_{(n)}\,.
\end{equation*} 
Using as above that $x\mapsto e^{-x}$ is Lipschitz continuous and bounded by $1$ on $[0,\infty)$, as well as the boundedness of $f$, we see that
\begin{align*}
&\Big|\E_{\lfloor un\rfloor}\Big[\fodd\Big(|X_{tn^2}|_{(n)} \Big)\exp\Big\{-\frac{\alpha}{n^{\beta}}\xi_{tn^2}(\{-1,0\})\Big\}]
- \E_{\lfloor un\rfloor}\Big[\fodd\Big(|X_{tn^2}|_{(n)} \Big)\Big]\Big|\\
&\leq\; C\times \E_{\lfloor un\rfloor}\Big[\frac{\xi_{tn^2}(\{-1,0\})}{n^{\beta}}\Big]\; \leq\; Cn^{1-\beta}\,.
\end{align*}
Here, we made use of Proposition \ref{prop:localtime} to arrive at the last estimate.  To conclude one may now proceed as in the proof of Lemma~\ref{lem:evenapprox}.
\end{proof}
\subsection{Convergence of the slow bond random walk at a fixed time}
We have gathered all ingredients to prove   the main result of this section, which immediately implies  Theorem~\ref{thmBE}.
\begin{theorem}
\label{thm:onedimdistributions}
Let $u>0$ and let $f\in \LSNOB$. By $\PSNOB$  denote the semigroup of the snapping out Brownian motion of parameter $\kappa=2\alpha$. 
Then, for all $t>0$, we have the estimates
\begin{itemize}
\item If $\beta\in[0, 1)$, then 
\begin{equation*}
\big|{\bf P}^n_t f  \bigP{\lfloor un \rfloor}{n} - {P}_t f (u)\big|\;\lesssim \;
\max\big\{n^{-1}, n^{\beta-1}\big\}\;=\; n^{\beta-1}\,. 
\end{equation*}
\item If $\beta= 1$, then  for all $\delta>0$,
\begin{equation*}
 \big|{\bf P}^n_t f \bigP{\lfloor un\rfloor}{n} - \PSNOB f (u)\big|\;\lesssim \;n^{-1/2 + \delta}\,.
\end{equation*} 
\item If $\beta\in(1,\infty]$, then
 \begin{equation*}
 \big|{\bf P}^n_t f \bigP{\lfloor un\rfloor}{n} - \E_u[f(|B_t|)]\big|
 \;\lesssim \; \max\{n^{-1},n^{1-\beta}\}\,.
\end{equation*}
\end{itemize} 
\end{theorem}
\begin{proof}
\textbf{Case $\beta\in[0,1)$}.
Writing ${\bf P}^n_t f \bigP{x}{n} = {\bf P}^n_t \fodd \bigP{x}{n} + {\bf P}^n_t \feven \bigP{x}{n}$,  we can apply Lemmas~\ref{lem:evenapprox} and \ref{lem:oddapprox} to infer that ${\bf P}^n_t f \bigP{\lfloor un\rfloor}{n}$ indeed converges to   $ P_t \fe (u) + P_t \fo (u) = P_t f(u) $ at the desired rate. \medskip

\textbf{Case $\beta=1$}.
Writing ${\bf P}^n_t f \bigP{x}{n} = {\bf P}^n_t \fodd \bigP{x}{n} + {\bf P}^n_t \feven \bigP{x}{n}$,   Lemmas \ref{lem:evenapprox} and \ref{lem:oddapprox} imply  that ${\bf P}^n_t f \bigP{\lfloor un\rfloor}{n}$  converges to   $ P_t \fe (u) + {Q}_t\fo  (u) $ at the desired rate.  

It therefore only remains to check that $  P_t \fe + {Q}_t \fo = \PSNOB f$, which can be verified via the following direct computation. Note that
\begin{equation}\label{id}
f(u)+f(-u)\;=\;f(|u|)+f(-|u|)\,,\qquad \forall\, u\in \bb R\,,
\end{equation}
 and recall \eqref{formula}. Then,
\begin{align*}
&{\bf P}_t \fe(u) + {Q}_t \fo(u)\\
&=\;\bb E_u\Big[\frac{f(B_t)+f(-B_t)}{2}\Big]+\bb E_u\Big[\frac{f(|B_t|)-f(-|B_t|)}{2}\exp\big\{-2\alpha L_t(0)\big\}\Big]\\
&=\;\bb E_u\Big[\frac{f(|B_t|)+f(-|B_t|)}{2}\Big]+\bb E_u\Big[\frac{f(|B_t|)-f(-|B_t|)}{2}\exp\big\{-2\alpha L_t(0)\big\}\Big]\\
&=\;\bb E_u\Big[\frac{1+\exp\big\{-2\alpha L_t(0)\big\}}{2}f(|B_t|)\Big]+\bb E_u\Big[\frac{1-\exp\big\{-2\alpha L_t(0)\big\}}{2}f(-|B_t|)\Big]\\
&=\;\PSNOB f(u)\,.
\end{align*}

 \textbf{Case $\beta\in(1,\infty]$}. It follows from Lemmas \ref{lem:evenapprox} and \ref{lem:oddapprox} that there exists a constant $C>0$ such that for all $t>0$ and all $n$ large enough
\begin{equation*}
 \big|{\bf P}^n_t f \bigP{\lfloor un\rfloor}{n} -\E_u[\fe(B_t)]-  \E_u[\fo(|B_t|)]\big|\;\lesssim\; \max\big\{n^{-1}, n^{1-\beta}\big\}\,.
\end{equation*}
To conclude the proof it therefore only remains to show that 
\begin{equation*}
\E_u\big[\fe(B_t)\big]\;= \;\E_u\big[\fe(|B_t|)\big]\,,
\end{equation*}
which follows by the observation \eqref{id}.
\end{proof}

\section{CLT  for finite-dimensional distributions}\label{s6}
In what follows, since there is no necessity to specify the precise value of $\beta$, we will use  $\BSB$ to denote the respective limiting process in Theorem \ref{thm21}, which can either be the BM, the snapping out BM or the reflected BM. The same applies for the notation $\XSB$ for the slow bond RW.

Fix $k\in \bb N$ and times $0=t_0 < t_1 < \ldots < t_k\leq 1$. We will  show in this section that 
\begin{equation}\label{eq:finitedim}
\frac1n \big(\XSB_{t_1n^2}, \XSB_{t_2n^2},\ldots, \XSB_{t_k n^2}\big)\Longrightarrow 
\big(\BSB_{t_1}, \BSB_{t_2},\ldots, \BSB_{t_k}\big),\,\quad \text{as }n\to\infty\,,
\end{equation}
where the arrow above denotes weak convergence. Let us introduce some notation. Given a process $Z$ and an independent copy $\hat{Z}$ of $Z$, denote by
 $\bb E^{\hat{Z},\,t}_{z}$ the expectation with respect to the process $\hat{Z}$  \textit{started at time} $t$ at the position $z$. This is not a standard notation, but it will be suitable for our purposes.
 
  For  $j\in\{0,\ldots, k-1\}$, and  $f\in \LSNOB$, we have by the Markov property that
\begin{equation}\label{eq:markovfinite}
\E_{x}\Big[f\Big(\pfrac{\XSB_{t_{j+1}n^2}- \XSB_{t_j n^2}}{n}\Big)\Big]
\;=\;\E_x\Big[\E_{\XSB_{t_j n^2}}^{\YSB,\, t_j n^2}\Big[f\Big(\pfrac{\YSB_{t_{j+1}n^2}- \XSB_{t_j n^2}}{n}\Big)\Big]\Big]\,,
\end{equation}
 Choose $x$ to be of the form $\lfloor un\rfloor$ with, let us say, $u>0$. We claim that, since the convergence in Theorem \ref{thm:onedimdistributions} is uniform in the starting point, it follows that the above converges to
\begin{equation}\label{eq:slowdifference}
\E_u \Big[f\big(\BSB_{t_{j+1}}-\BSB_{t_j}\big)\Big].
\end{equation}
Indeed, we have on the one hand that
\begin{equation*}
\Big|\E_x\Big[\E^{\YSB, \,t_jn^2}_{\XSB_{t_j n^2}}\Big[f\Big(\pfrac{\YSB_{t_{j+1}n^2}- \XSB_{t_j n^2}}{n}\Big)\Big]\Big] -
\E_{x/n}\Big[\E^{\BSB,\, t_j}_{\XSB_{t_j n^2}/n}\Big[f\Big(\BSB_{t_{j+1}}- \pfrac{\XSB_{t_j n^2}}{n}\Big)\Big]\Big]\Big|
\end{equation*}
converges to zero, which is a consequence of the uniformity in Theorem \ref{thm:onedimdistributions} alluded to above.
On the other hand, denoting by $\BSBhat$ an independent copy of $\BSB$, we affirm that
\begin{equation}\label{eq:64}
\Big|\E_{\frac{x}{n}}\Big[\E^{\BSB,\, t_j}_{\XSB_{t_j n^2}/n}\Big[f\Big(\BSB_{t_{j+1}}- \pfrac{\XSB_{t_j n^2}}{n}\Big)\Big]\Big] - \E_{\frac{x}{n}}\Big[\E^{\BSBhat, \, t_j}_{\BSB_{t_j }}\Big[f\Big(\BSBhat_{t_{j+1}}- \BSB_{t_j}\Big)\Big]\Big]\Big|
\end{equation}
goes to zero. To see this, define  the function $g$  by
\begin{equation*}
g(x)\;=\; \E_x^{\BSBhat,\, t_j}\big[f(\BSBhat_{t_{j+1}}- x)\big]\,.
\end{equation*}
which belongs to $\LSNOB$ due to Corollary~\ref{cor:Lipschitz}.
Then, \eqref{eq:64} can be written as 
\begin{equation*}
\Big|\E_{\frac{x}{n}}\Big[g\Big(\pfrac{\XSB_{t_j n^2}}{n}\Big)\Big] - \E_{\frac{x}{n}}[g\big(\BSB_{t_j})\big]\Big|\,,
\end{equation*}
which proves the claim by applying Theorem~\ref{thm:onedimdistributions}. 
Now, note that $\XSB$ has independent increments, so that the above arguments yields that
\begin{equation}\label{eq:finitedimdif}
\begin{split}
\frac1n (\XSB_{t_1n^2},\,& \XSB_{t_2n^2}-\XSB_{t_1n^2},\ldots, \XSB_{t_kn^2}-\XSB_{t_{k-1}n^2})\\
&\Rightarrow (\BSB_{t_1}, \BSB_{t_2}-\BSB_{t_1},\ldots, \BSB_{t_k}-\BSB_{t_{k-1}})\,,
\end{split}
\end{equation}
as $n$ tends to infinity. The desired convergence of the finite dimensional distributions now follows, since \eqref{eq:finitedim} is the image of the left hand side of \eqref{eq:finitedimdif} under a linear map.

\section{Tightness in the \texorpdfstring{$J_1$}{J1}-topology}\label{sec:tight}
In this section we show that the sequence  $\{n^{-1}\XSB_{tn^2}:t\in[0,1]\}$ is tight in the $J_1$-topology of Skorohod  of $\mathscr{D}([0,1], \bb R)$. To do so, we make use of the following criterion that can be found in \cite[Theorem 13.5]{Bill}.
\begin{proposition}
Consider a sequence $(X^n)_{n\in\bb N}$ and a process $X$ in $\mathscr{D}([0,1], \bb R)$. Assume that the finite dimensional distributions of $(X^n)_{n\in\bb N}$ converge to those of $X$, and assume that $X$ is almost surely continuous at $t=1$. Moreover assume that there are $\beta\geq 0$, $\alpha >1/2$ and a non-decreasing continuous function $F$ such that for all $r\leq s\leq t$, all $n\geq 1$, and all $x\in \bb Z$,
\begin{equation}\label{eq:momentcond}
\E_x \Big[|X_s^n-X_r^n|^{2\beta}\,|X_t^n-X_s^n|^{2\beta}\Big]\;\leq\; \big[F(t)-F(r)\big]^{2\alpha}\,.
\end{equation}
Then, the sequence $(X^n)_{n\in\bb N}$ converges to $X$ in the $J_1$-topology of Skorohod of\break $\mathscr{D}([0,1], \bb R)$.
\end{proposition}
As a consequence of the above result we only need to establish the 
the moment condition \eqref{eq:momentcond}. We claim that it is enough to show that there is a constant $C$ such that for any pair of times $0\leq s \leq t$, and any starting point $x$, the following inequality holds:
\begin{equation}\label{eq:realmomentcond}
\E_x\Big[\Big|\frac{\XSB_{tn^2}}{n}-\frac{\XSB_{sn^2}}{n}\Big|^2\Big]\;\leq\; C|t-s|\,.
\end{equation}
Indeed assume that \eqref{eq:momentcond} holds and let $r\leq s\leq t$. Then the Markov property applied at time $sn^2$ yields
\begin{align*}
&\frac{1}{n^4}\E_x\Big[|\XSB_{sn^2}-\XSB_{rn^2}|^2\,|\XSB_{tn^2}-\XSB_{sn^2}|^2\Big]\\
&=\; \frac{1}{n^4}\E_x\Big[|\XSB_{sn^2}-\XSB_{rn^2}|^2\,
\E_{\XSB_{sn^2}}\big[|\XSB_{tn^2}-\XSB_{sn^2}|^2\big]\Big]\\
& \leq \;C^2|t-s|\, |s-r|\;\leq \; C^2|t-r|^2\,,
\end{align*}
hence the claim follows.
To establish \eqref{eq:realmomentcond}, recall that Dynkin's formula yields that for any function $f$ in the domain of $\mathsf{L}_n$, there is a martingale $\mathscr{M}(f)$ such that
\begin{equation}\label{eq:dynkin}
f\Big(\frac{\XSB_{tn^2}}{n}\Big)\;=\;f\Big(\frac{\XSB_0}{n}\Big) + \int_0^{tn^2} \mathsf{L}_n f\Big(\frac{\XSB_s}{n}\Big)\, ds + \mathscr{M}_{tn^2}(f)\,.
\end{equation}
Our case is the case in which $f$ is the identity. Note that in this case the definition of $\mathsf{L}_n$ implies that
\begin{equation*}
\mathsf{L}_n f\Big(\frac{x}{n}\Big)\;=\;  \frac1n\big[\xi_{x,x+1}-\xi_{x,x-1}\big]\;=\;\frac1n\times\begin{cases}
\frac12 -\frac{1}{2n^\beta}, &\text{if }x=0,\\
\frac{1}{2n^\beta}-\frac12, &\text{if }x=-1,\\
0, &\text{otherwise.}
\end{cases}
\end{equation*}
Denoting by $\ell$ the local time of  $\XSB$, the above considerations then show that the right hand side of \eqref{eq:dynkin} equals 
\begin{equation*}
f\Big(\frac{\XSB_0}{n}\Big) + \frac1n\Big[\frac12-\frac{1}{2n^\beta}\Big]\big[\ell_{tn^2}(0)-\ell_{tn^2}(-1)\big] +\mathscr{M}_{tn^2}(f)\,.
\end{equation*}
Thus, to show \eqref{eq:realmomentcond} it is enough to bound the second moment of 
\begin{equation*}
\frac1n\big[\ell_{sn^2,tn^2}(0)-\ell_{sn^2,tn^2}(-1)\big]\quad \text{and}\quad
\big[\mathscr{M}_{tn^2}(f)- \mathscr{M}_{sn^2}(f)\big]\,.
\end{equation*}
Here, we used the notation $\ell_{s,t}$ to denote the local time of $\XSB$ between times $s$ and $t$. 
We first analyse the local time term above. To that end, we note that Proposition \ref{prop:lumping} and the discussion following it yield a coupling between $(\ell_{tn^2}(0) + \ell_{tn^2}(-1))_{t\geq 0}$ and $(\xi_{tn^2}(0)+\xi_{tn^2}(-1))_{t\geq 0}$ under which these processes are equal. We recall that $\xi$ denotes the local time process of the usual\break continuous-time symmetric random walk.
Since
\begin{equation*}
|\ell_{sn^2,tn^2}(0)-\ell_{sn^2,tn^2}(-1)|\;\leq\; |\ell_{sn^2,tn^2}(0)+\ell_{sn^2,tn^2}(-1)|
\end{equation*}
and $x\mapsto x^2$ is a monotone function of the modulus of $x$, we see that it is sufficient to estimate the second moment of the sum of the respective local times between times $sn^2$ and $tn^2$. However, by the coupling just mentioned it is sufficient to estimate
\begin{equation*}
\frac{1}{n^2}\E_x \Big[(\xi_{sn^2,tn^2}(0) + \xi_{sn^2,tn^2}(-1))^2\Big]\,,
\end{equation*}
and we obtain the desired estimate as a consequence of Proposition \ref{prop:localtime}. 
We turn to the analysis of the martingale term. To that end we apply the following version of the Burkholder-Davis-Gundy inequality:
\begin{theorem}\label{thm:BDG}
Let $M$ be a c\`adl\`ag square integrable martingale. For any $p>0$ there exists a constant $C=C(p)>0$ such that for all $T>0$,
\begin{equation*}
\E \Big[\sup_{0\leq t\leq T} |M_t|^p\Big]
\;\leq\; C\,\E \Big[ [M,M]_T^{p/2}\Big]\,.
\end{equation*}
\end{theorem}
Note that $\mc M_{tn^2}= \mathscr{M}_{tn^2}(f)- \mathscr{M}_{sn^2}(f)$ is also a martingale in $t\geq s$ whose optional quadratic variation is given by
\begin{equation*}
[\mc M, \mc M ]_{tn^2}\;=\; \frac{1}{n^2}\sum_{sn^2\leq r\leq tn^2}|\Delta_r \XSB|^2\,,
\end{equation*}
where $\Delta_r \XSB$ denotes the size of the jump of $\XSB$ at time $r$. Note that $\XSB$ only does jumps of size one, so that the above is $1/n^2$ times the number of jumps in the time interval $[sn^2,tn^2]$. However, since for $\beta\geq 0$ we always have that $\alpha/2n^{\beta}\leq \max\{1/2,\alpha/2\}$, it readily follows that the number of jumps of $\XSB$ in the time interval $[sn^2,tn^2]$ is stochastically dominated by the number of jumps of a continuous-time simple symmetric random walk jumping at rate $\max\{1,\alpha\}$, i.e., by $N_{(t-s)n^2}(m)$, where $m= \max\{1,\alpha\}$ and $N(m)$ is a Poisson process with rate $m$. We can now conclude the proof using that
\begin{equation*}
\E\big[N_{(t-s)n^2}(m)\big] \;=\; (t-s)n^2m\,.
\end{equation*}

\appendix
\section{Auxiliary tools}\label{appendixA}

Here we resume the idea from \cite{fgn2} on how to obtain the explicit solution of PDE  \eqref{pdeRobin}.
\begin{equation}\label{pde11}
\begin{cases}
\partial_t \rho=\frac{1}{2}\Delta \rho, \,\,\,u\neq 0\\ 
\partial_u \rho(t,0^{+})=\partial_u \rho(t,0^{-})=\frac{\kappa}{2}[\rho(t,0^{+})-\rho(t,0^{-})]\\
\rho(0,u)=f(u).
\end{cases}
\end{equation}
Denote by $T^{\kappa}_t f(u)$ the solution of above, where $f$ is  the initial condition and denote  by $\fe(u)$ and $\fo(u)$ its even and odd parts, respectively. By linearity, the solution of \eqref{pde11} may be written as the sum of  $T^{\kappa}_t \fe(u)$  and $T^{\kappa}_t \fo(u)$. Since the PDE \eqref{pde11} preserves parity, we conclude that  $T^{\kappa}_t \fe(u)$ is solution of 
\begin{equation}\label{pde12}
\begin{cases}
\partial_t \rho=\frac{1}{2}\Delta \rho\\ 
\rho(0,u)=\feven(u)\\
\partial_u \rho(t,0^{+})=\partial_u \rho(t,0^{-})=0,
\end{cases}
\end{equation}
which boundary condition can be dropped due to the fact that $\fe$ is an even function. That is, $T^{\kappa}_t \fe(u)$ is simply the solution of usual heat equation
\begin{equation}\label{pde12b}
\begin{cases}
\partial_t \rho=\frac{1}{2}\Delta \rho\\ 
\rho(0,u)=\fe(u).\\
\end{cases}
\end{equation}
which solution is given by the classical formula $$\rho(t,u)\;=\;\int_{\R}f_{\text{even}}(u-y)\dfrac{e^{-y^2/2t}}{\sqrt{2\pi t}}dy\,.$$
On the other hand,  again by  preservation of parity,  we can deduce that\break $T^{\kappa}_t \fo(u)$ is given by
\begin{equation}\label{pde13}
\begin{cases}
\partial_t \rho=\frac{1}{2}\Delta \rho,\,\,\,u>0 \\ 
\rho(0,u)=\fo(u),\,\,\,u>0\\
\partial_u \rho(t,0^{+})=\kappa \rho(t,0^{+}),\,\,\,t>0.
\end{cases}
\end{equation}
on the positive half line, with analogous definition on the negative half line.
The standard technique to solve the \eqref{pde13} is to define 
\begin{equation}\label{ODE}
v(t,u)\;:=\;\kappa \rho(t,u)-\partial_u \rho(t,u)\,,
\end{equation}
which will be solution of 
\begin{equation}\label{pde14}
\begin{cases}
\partial_t v=\frac{1}{2}\Delta v,\,\,\,u>0 \\ 
v(0,u)=\kappa \fo (u)-f'_{\text{odd}}(u)\,\,\,u>0\\
v(t,0)=0.
\end{cases}
\end{equation}
with analogous definition for the negative half line.
Note that the equation above has Dirichlet boundary conditions, which can easily solved by the \textit{image method}. Once we have the expression for $v$, solving the linear ODE \eqref{ODE} gives us the expression for $T^{\kappa}_t \fo(u)$.

\begin{proposition}
The random variables $L(x,t)$ and $L(x\sqrt{n},tn)/\sqrt{n}$ have the same distribution.
\end{proposition}
\begin{proof}
Doing the changing of variables  $u =sn$, we get
\begin{align*}
L(x,t) &\;=\; \lim_{\eps\searrow 0}\frac{1}{2\eps} \int_0^t \mathds{1}_{B_s\in (x-\eps,x+\eps)}ds \;=\; \lim_{\eps\searrow 0}\frac{1}{2\eps n} \int_0^{tn} \mathds{1}_{B_{u/n} \in (x-\eps,x+\eps)}\frac{du}{n} \\
&  \;=\; \lim_{\eps\searrow 0}\frac{1}{2\eps n} \int_0^{tn} \mathds{1}_{\sqrt{n}B_{u/n}\in (x\sqrt{n}-\eps \sqrt{n},x\sqrt{n}+\eps\sqrt{n})}\frac{du}{n}\,. 
\end{align*}
Due to the BM's scaling invariance, the last expression is equal in law to
\begin{align*}
\lim_{\eps\searrow 0}\frac{1}{2\eps n} \int_0^{tn} \mathds{1}_{B_u\in (x\sqrt{n}-\eps \sqrt{n},x\sqrt{n}+\eps\sqrt{n})}\frac{du}{n}\;=\; \frac{L(x\sqrt{n},tn)}{\sqrt{n}}\,.
\end{align*}
\end{proof}

The next result is probably standard, however we were not able to find it in the literature, so we provide a proof. Recall that $\xi_{tn^2}(0)$ denotes the local time of simple random walk at the origin.
\begin{proposition}
\label{prop:localtime}
Let $p\in \bb N$, then for all $t>0$ there exists a constant $C>0$ such that for all $n\in \bb N$ and all $x\in \bb Z$,
\begin{equation*}
\E_x\big[(\xi_{sn^2,tn^2}(0))^p\big]\;\lesssim\; |t-s|^{\frac{p}{2}}n^p\,.
\end{equation*}	
\end{proposition}
\begin{proof}
For simplicity we prove the result only for $s=0$, however since our estimates are uniform in the starting point, the general case is a straightforward consequence.
First note that a change of variables yields that
\begin{equation*}
\xi_{tn^2}(0)\;=\; \int_0^{tn^2}\mathds{1}_{\{X_s=0\}}\, ds\;=\; 
n^2\int_0^{t} \mathds{1}_{\{X_{sn^2}=0\}}\, ds\,.
\end{equation*}
We then see that
\begin{equation}\label{eq:localtimeexpansion}
\E\big[(\xi_{tn^2}(0))^p\big] \;=\; n^{2p}m! \int_0^{t}\,ds_1\, \int_{s_1}^{t}\,ds_2\cdots\int_{s_{m-1}}^{t}\, ds_m
\prod_{i=1}^{m}p_{(s_i-s_{i-1})n^2}(0)\,,
\end{equation}
where we set $s_0=0$. We apply now the local central limit theorem \cite[Theorem 2.5.6]{Lawler}, which states that there is a constant $c$ such that for all $t$ and all~$n$
\begin{equation*}
np_{tn^2}(0)\;\leq\; \frac{c}{\sqrt{t}}\,.
\end{equation*}
Plugging this estimate into \eqref{eq:localtimeexpansion} we may now finish the proof. 
\end{proof}

Next we furnish a short proof of Lemma~\ref{lem:evenapprox}.
\begin{proof}[Proof of Lemma~\ref{lem:evenapprox}]
By equation \eqref{eq47} and the Lipschitz continuity of $u\mapsto  P_t \fe (u)$ provided by Corollary~\ref{cor:Lipschitz}  it is enough to prove \eqref{eq:evenapprox} with ${P}_t\fe (u)$  replaced by ${P}_t\feven \bigP{\lfloor un\rfloor }{n}$.  Moreover, to simplify notation we denote by $un\in\bb Z$ its integer part. We can now write
\begin{equation}\label{eq:semigroupeven}
{\bf P}^n_t \feven \bigP{\lfloor un\rfloor}{n}\;=\; \sum_{z\in \frac{1}{n}\bb Z}
\feven(z) p_{tn^2}(n(u-z))\,.
\end{equation}
We now apply the local central limit theorem, \cite[Theorem 2.3.11]{Lawler} which states that for $x\in \frac{1}{n}\bb Z$,
\begin{equation*}
np_{tn^2}(nx)\;=\; K_t(x)\exp\Bigg\{{O}\Big(\frac{1}{tn^2} +\frac{\|nx\|^4}{(tn^2)^3}\Big)\Bigg\}\,,
\end{equation*}
where $K_t$ denotes the usual heat kernel.
We use this estimate in \eqref{eq:semigroupeven} for all $z\in \frac{1}{n}\bb Z$ such that $|n(u-z)|\leq n^{5/4}$. Since there exists a constant $C>0$ such that for all $x\in [0,1)$ we have the estimate $|e^{x}-1|\leq C|x|$ the above states that $|np_{tn^2}(n(u-z))-K_t(u-z)|\leq \frac{C}{n}$ for the range of $z$'s just mentioned.
Moreover, note that 
\begin{equation*}
\Big|\sum_{\topo{z\in \frac{1}{n}\bb Z\,:}{|n(u-z)|\geq n^{5/4}}} \feven(z) p_{tn^2}(n(u-z))\Big|
\;\leq\;  \|f\|_{\text{L}}\;\bb P_0\big[\,|X_{tn^2}|\geq n^{5/4}\,\big]\,,
\end{equation*}
and by \cite[Proposition 2.1.2 (b)]{Lawler} we see that the above is bounded by\break $C_1e^{-C_2 n^{1/8}}$, for some constants $C_1$ and $C_2$.
The proof may now be finished by using the above approximation of the continuous heat kernel by the discrete one and by a standard Riemann sum approximation. We omit the details.
\end{proof}

\section*{Acknowledgements}
T. F. was supported by a project Jovem Cientista-9922/2015, FAPESB-Brazil and by National Council for Scientific and Technological Development (CNPq) through a Bolsa de Produtividade. D. E. gratefully acknowledges financial support 
from the National Council for Scientific and Technological Development - CNPq via a 
Universal grant 409259/2018-7. 
D. S.  would like to thank CAPES for a PhD scholarship, which supported his research.

\bibliography{bibliografia}

\begin{thebibliography}{10}

\bibitem{Amir}
M.~Amir.
\newblock Sticky {B}rownian motion as the strong limit of a sequence of random
  walks.
\newblock {\em Stochastic Process. Appl.}, 39(2):221--237, 1991.

\bibitem{araujoguine}
A.~Araujo and E.~Gin\'{e}.
\newblock {\em The central limit theorem for real and {B}anach valued random
  variables}.
\newblock John Wiley \& Sons, New York-Chichester-Brisbane, 1980.
\newblock Wiley Series in Probability and Mathematical Statistics.

\bibitem{Bill}
P.~Billingsley.
\newblock {\em Convergence of Probability Measures}.
\newblock John Wiley and Sons, 2nd edition, 1999.

\bibitem{Borodin}
A.~N. Borodin.
\newblock Brownian local time.
\newblock {\em Uspekhi Mat. Nauk}, 44(2(266)):7--48, 1989.

\bibitem{Csaki2009}
E.~Cs\'{a}ki, M.~Cs\"{o}rg\H{o}, A.~F\"{o}ldes, and P.~R\'{e}v\'{e}sz.
\newblock Random walk local time approximated by a {B}rownian sheet combined
  with an independent {B}rownian motion.
\newblock {\em Ann. Inst. Henri Poincar\'{e} Probab. Stat.}, 45(2):515--544,
  2009.

\bibitem{OnBest}
M.~Cs\"{o}rg\H{o} and L.~Horváth.
\newblock {On best possible approximations of local time}.
\newblock {\em Statistics \& Probability Letters}, 8(4):301--306, 1989.

\bibitem{Durrett}
R.~Durrett.
\newblock {\em {Probability: Theory and Examples}}, volume~31 of {\em Cambridge
  Series in Statistical and Probabilistic Mathematics}.
\newblock Cambridge University Press, Cambridge, fourth edition, 2010.

\bibitem{efgnt}
D.~Erhard, T.~Franco, A.~Neumann, P.~Gon\c{c}alves, and M.~Tavares.
\newblock Non-equilibrium fluctuations for the {SSEP} with a slow bond.
\newblock {\em Accepted for publication in Ann. Inst. H. Poincar\'e Prob. and
  Stat.}, 2019.

\bibitem{fgn1}
T.~Franco, P.~Gon\c{c}alves, and A.~Neumann.
\newblock Hydrodynamical behavior of symmetric exclusion with slow bonds.
\newblock {\em Ann. Inst. H. Poincar\'{e} Probab. Statist.}, 49(2):402--427,
  2013.

\bibitem{fgn2}
T.~Franco, P.~Gon\c{c}alves, and A.~Neumann.
\newblock Phase transition in equilibrium fluctuations of symmetric slowed
  exclusion.
\newblock {\em Stoch. Proc. Appl.}, 123(12):4156--4185, 2013.

\bibitem{fgn3}
T.~Franco, P.~Gon\c{c}alves, and A.~Neumann.
\newblock {Phase transition of a heat equation with Robin's boundary conditions
  and exclusion process}.
\newblock {\em Trans. Amer. Math. Soc.}, 367:6131--6158, 2015.

\bibitem{FGSCMP}
T.~Franco, P.~Gon\c{c}alves, and M.~Simon.
\newblock Crossover to the stochastic {B}urgers equation for the {WASEP} with a
  slow bond.
\newblock {\em Comm. Math. Phys.}, 346(3):801--838, 2016.

\bibitem{FN}
T.~Franco and A.~Neumann.
\newblock Large deviations for the exclusion process with a slow bond.
\newblock {\em Ann. Appl. Probab.}, 27(6):3547--3587, 2017.

\bibitem{Grebenkov}
D.~S. Grebenkov.
\newblock Partially reflected {B}rownian motion: a stochastic approach to
  transport phenomena.
\newblock In {\em Focus on probability theory}, pages 135--169. Nova Sci.
  Publ., New York, 2006.

\bibitem{NMR_survey}
D.~S. Grebenkov.
\newblock {NMR} survey of reflected {B}rownian motion.
\newblock {\em Rev. Mod. Phys.}, 79:1077--1137, 2007.

\bibitem{kl}
C.~Kipnis and C.~Landim.
\newblock {\em {Scaling limits of interacting particle systems}}, volume 320 of
  {\em {Grundlehren der mathematischen Wissenschaften}}.
\newblock Springer-Verlag Berlin Heidelberg, 1st edition, 1999.

\bibitem{Lawler}
G.~F. Lawler and V.~Limic.
\newblock {\em Random walk: a modern introduction}, volume 123 of {\em
  Cambridge Studies in Advanced Mathematics}.
\newblock Cambridge University Press, Cambridge, 2010.

\bibitem{Lejay}
A.~Lejay.
\newblock The snapping out {B}rownian motion.
\newblock {\em Ann. Appl. Probab.}, 26(3):1727--1742, 2016.

\bibitem{Levy39}
P.~L\'{e}vy.
\newblock Sur certains processus stochastiques homog\`enes.
\newblock {\em Compositio Math.}, 7:283--339, 1939.

\bibitem{Levy65}
P.~L\'{e}vy.
\newblock {\em Processus stochastiques et mouvement brownien}.
\newblock Suivi d'une note de M. Lo\`eve. Deuxi\`eme \'{e}dition revue et
  augment\'{e}e. Gauthier-Villars \& Cie, Paris, 1965.

\bibitem{Revesz}
P.~R\'{e}v\'{e}sz.
\newblock Local time and invariance.
\newblock In {\em Analytical methods in probability theory ({O}berwolfach,
  1980)}, volume 861 of {\em Lecture Notes in Math.}, pages 128--145. Springer,
  Berlin-New York, 1981.

\bibitem{RevuzYor}
D.~Revuz and M.~Yor.
\newblock {\em Continuous martingales and {B}rownian motion}, volume 293 of
  {\em Grundlehren der Mathematischen Wissenschaften [Fundamental Principles of
  Mathematical Sciences]}.
\newblock Springer-Verlag, Berlin, third edition, 1999.

\bibitem{Takacs}
L.~Tak\'{a}cs.
\newblock On the local time of the {B}rownian motion.
\newblock {\em Ann. Appl. Probab.}, 5(3):741--756, 1995.

\bibitem{Trotter}
H.~F. Trotter.
\newblock A property of {B}rownian motion paths.
\newblock {\em Illinois J. Math.}, 2:425--433, 1958.

\bibitem{HK}
R.~van~der Hofstad and M.~Keane.
\newblock An elementary proof of the hitting time theorem.
\newblock {\em Amer. Math. Monthly}, 115(8):753--756, 2008.

\end{thebibliography}
\bibliographystyle{plain}

\end{document}